\documentclass[a4]{article}
\usepackage{enumitem,centernot,comment}
\usepackage{graphicx,xcolor}
\usepackage{url}
\usepackage{amssymb,amsmath,amsthm}
\usepackage{authblk}
\usepackage{colortbl}
\usepackage{bm}
\usepackage{multirow}
\usepackage[style=numeric-comp,maxbibnames=99,bibencoding=utf8,firstinits=true,backend=biber]{biblatex}
\usepackage{savesym}
\savesymbol{hyperpage}
\usepackage{hyperref}
\restoresymbol{HR}{hyperpage}
\usepackage{subcaption}
\usepackage{pgfplots,pgfplotstable}
\usetikzlibrary{plotmarks}
\usepackage{bbm}
\usepackage{tikz-cd}
\pgfplotsset{compat=1.18}
\graphicspath{{figures/}}
 \usepackage[normalem]{ulem}
  \newcounter{corr}
  \definecolor{violet}{rgb}{0.580,0.,0.827}
  \newcommand{\corr}[3]{\typeout{Warning : a correction remains in page \thepage}
    \stepcounter{corr}        
  	      {\color{blue}\ifmmode\text{\,\sout{\ensuremath{#1}}\,}\else\sout{#1}\fi}
                {\color{red}#2}
                {\color{violet} #3}
  }
  
  \newcounter{upd}
  \definecolor{green}{rgb}{0.0, 0.5, 0.0}
  \definecolor{brown}{rgb}{0.8, 0.33, 0.0}
  \newcommand{\upd}[3]{\typeout{Warning : an update remains in page \thepage}
  	\stepcounter{upd}        
  	{\color{blue}\ifmmode\text{\,\sout{\ensuremath{#1}}\,}\else\sout{#1}\fi}
  	{\color{green}#2}
  	{\color{brown} #3}}

\hypersetup{
    colorlinks,
    linkcolor={red!50!black},
    citecolor={blue!50!black},
    urlcolor={blue!80!black},
  linkbordercolor={red!50!black},
}

\def\R{\mathbb{R}}
\def\N{\mathbb{N}}

\def\eui{e_u^{n, i}}

\def\eum{e_u^{n, i-1}}

\def\evi{e_v^{n, i}}

\def\evm{e_v^{n, i-1}}

\def\rzeta{\zeta_\varepsilon}

\bibliography{references} 

\newtheorem{thm}{Theorem}
\numberwithin{thm}{section}
\newtheorem{lemma}[thm]{Lemma}

\newtheorem{remark}[thm]{Remark}
\newtheorem{definition}[thm]{Definition}
\newtheorem{algo}[thm]{Algorithm}

\newcommand{\lip}{\left\langle}
\newcommand{\rip}{\right\rangle}

\newcommand{\email}[1]{\href{mailto:#1}{#1}}
\newcommand{\mb}{\mathbb}
\newcommand{\mc}{\mathcal}
\newcommand{\xdo}[1][]{X_{\mc{D}#1,0}}
\newcommand{\pd}{\Pi_{\mc{D}}}
\newcommand{\dd}{\nabla_{\mc{D}}}
\newcommand{\bs}{\boldsymbol}
\newcommand{\lkl}{\mathcal{L}(\mc{K}(\Theta),L^{2}(\Theta))}
\newcommand{\mv}[1][u]{\mathbf{#1}}
\newcommand{\bea}[1][]{\begin{equation}\label{#1}\begin{aligned}}
\newcommand{\eea}{\end{aligned}\end{equation}}
\newcommand{\beas}{\begin{equation*}\begin{aligned}}
\newcommand{\eeas}{\end{aligned}\end{equation*}}
\newcommand{\NN}{\mathbb N}
\newcommand{\RR}{\mathbb R}

\usepackage{soul}

\begin{document}

\title{Efficient iterative linearised solvers for numerical approximations of stochastic Stefan problems}
\date{August 2025}
\author[1,4]{Muhammad Awais Khan\footnote{\email{akhan@gudgk.edu.pk}}}
\author[1,2]{J\'er\^ome Droniou\footnote{\email{jerome.droniou@umontpellier.fr}}}
\author[1]{Kim-Ngan Le\footnote{\email{ngan.le@monash.edu}}}
\author[3]{Sorin Pop\footnote{ \email{sorin.pop@uhasselt.be}}}
\affil[1]{School of Mathematics, Monash University, Australia}
\affil[2]{IMAG, Univ. Montpellier, CNRS, Montpellier, France}
\affil[3]{Hasselt University, Campus Diepenbeek, Agoralaan Gebouw D, 3590 Diepenbeek, Belgium}
\affil[4]{Ghazi University, Dera Ghazi Khan, Pakistan}

\maketitle

\begin{abstract}
We present iterative solvers to approximate the solution of numerical schemes for stochastic Stefan problems. After briefly talking about the convergence results, we tackle the question of efficient strategies for solving the nonlinear equation associated with this scheme. We explore several approaches, from a standard Newton technique to linearised solvers. The latter offer the advantage of using the same coefficient matrix of the linearised system in each nonlinear iteration, for all time steps, and across all realisations of the Brownian motions. As a consequence, the system can be factorised once and for all. Although the linearised approach has a slower convergence rate, our sensitivity analysis and the use of adaptive tolerance in both deterministic and stochastic cases provide valuable insights for choosing the most effective solver across various scenarii.
\end{abstract}
\section{Introduction}
Stochastic partial differential equations (SPDEs) is an important research field in science and engineering and has numerous applications. Some of the influential applications are: super-processes in biological systems \cite{parisi1981perturbation}, the KPZ equation on dynamical scaling of a growing interface \cite{kardar1986dynamic} and the stochastic quantisation in Euclidean quantum field theory \cite{jona1985stochastic}. 

A great effort has been made to investigate the subject theoretically as well as numerically in the last fifty years. The theoretical foundations concerning the existence, uniqueness, and regularity of solutions to stochastic partial differential equations (SPDEs) have been extensively developed; see, for instance, \cite{Walsh1986introduction}. Parallel to these developments, considerable effort has been devoted to the numerical analysis of stochastic parabolic PDEs. Various numerical approaches have been proposed, including Galerkin methods \cite{grecksch1996time,hausenblas2003semilinear,gyongy2009numerical}, implicit time-stepping schemes \cite{gyongy1997implicit,kamrani2018numerical} and finite element methods \cite{allen1998finite,walsh2005finite}. Recent advancements include the numerical treatment of SPDEs such as the stochastic Allen–Cahn equation with multiplicative noise \cite{majee2018numerical}, stochastic p-Laplace systems \cite{breit2021numerical,diening2023numerical}, and stochastic total variation flow \cite{bavnas2021numerical}.
The gradient discretisation method has also been employed for stochastic evolution equations governed by Leray–Lions operators \cite{droniou2020design} and for stochastic non-Newtonian Stokes flows with transport noise \cite{droniou2024stokes}.

The deterministic Stefan problem characterises the phase transition between two thermodynamical states, such as from solid to liquid, within a non-mobile medium. A fundamental difference between deterministic and stochastic Stefan problems (SSP) lies in the nature of their solutions. While the deterministic formulation yields a unique trajectory for the temperature field and the moving interface, the stochastic counterpart characterises the solution as a probability distribution, capturing the range of possible system behaviours under uncertainty. This probabilistic framework is crucial for modelling and analysing phenomena where randomness plays a significant role -- such as alloy solidification in fluctuating environments or heat transfer in complex, heterogeneous media.

Solving stochastic partial differential equations presents significant computational challenges due to the need for simulating solutions across a finite sample set of Brownian motions. The number of Brownian motions required to approximate the solution is closely linked to the intensity of the noise in the system. To obtain an enhanced resolution of the solution, a sufficiently large sample size is often necessary. However, increasing the number of Brownian motions leads to a considerable rise in computational cost. 

A comprehensive numerical analysis of stochastic Stefan problem has been conducted recently, in \cite{DRONIOU2024114}, where the mass-lumped $\mathbb{P}^1$ (MLP1) finite element and Hybrid mimetic mixed (HMM) methods are used for the numerical simulations. Since the corresponding equations are nonlinear, the Newton method was used to compute approximate solutions to these schemes. It is observed during the simulations that the cost of quadratic convergence is quite expensive in terms of CPU time. This nonlinear solver moreover requires to re-assemble the Jacobian at each time step and for each realisation of the Brownian motion, which leads to a size-able overhead in computational cost. As an alternative, one can consider other iterative schemes that have been designed for numerical approximations of deterministic Stefan problems, like the so-called L-Scheme \cite{L_Scheme, List}, or a modified variant of it \cite{MITRA_Mscheme20191722}. In the present work, we explore the  efficiency of such solvers for stochastic models.

Consider the stochastic Stefan problem:
\begin{equation}\label{eq:sp}
\begin{cases}
du-\text{div}[\nabla\zeta(u)]dt=f(\zeta(u)) dW_{t}\quad \text{in}\quad \Theta_{T}:=(0,T)\times \Theta,\\
u(0,\cdot)=u_{0} \quad \text{in}\quad \Theta,\\
\zeta(u)=0\quad \text{on}\quad (0,T)\times \partial\Theta,
\end{cases}
\end{equation}
where $T>0$, $\Theta$ is an open-bounded domain in $\mathbf{R}^{d}\left( d\ge 1\right)$ and $W$ is a $\mathcal{Q}$-Wiener process. The nonlinear term $\zeta$ is a globally non-decreasing Lipschitz continuous function passing through the origin and coercive.

For $\omega$ in a 1-probability set $\Omega$, a solution to the above problem is a stochastic process $u(t,x,\omega)$ which is numerically approximated using a Monte Carlo procedures. That is, having $R$ samples $(\omega_i)_{i=1,\ldots,R}$, a numerical method will typically compute
$$\mathbb{E}_{R}[u_h(x_j,t_n)]=\frac{1}{R}\sum^{R}_{i=1} u_h(x_j,t_n,\omega_i),$$ 
where $\{x_j\}_{j=1}^M$ are the discretised spatial points and $\{t_n\}_{n=1}^{N}$ is the time discretisation.

Observe that, at each time step and for each realisation $W_t(\omega)$, it is necessary to solve a nonlinear system. When doing so, a linearisation technique such as the Newton method, or the L-Scheme 
can be employed, which ultimately reduces the fully discrete, nonlinear problems to sequences of linear systems $A_h^{n, i} U_h^{n, i} = b_h^{n, i}$, with $n$ and $h$ indicating the time step and the spatial mesh, and $i$ the iteration index. For any of the schemes below, the unknown vector is $U_h^{n, i}$. The matrix $A_h^{n,i}$ and the vector $b_h^{n,i}$ are computed depending on the chosen scheme. For any of the schemes considered below, the vector $b_h^{n,i}$ is computed at every iteration and for each realisation of the Brownian motion. In the case of the Newton method, the same holds for the matrix $A_h^{n, i}$, implying that this matrix is assembled $R\times N\times I$ times, where $I$ is the number of iterations required to approximate the nonlinear system and $N$ the number of time steps. The other  methods discussed here use the idea of the L-Scheme, where the linearised operator does not change during iterations. Consequently, the matrix $A_h^{n, i}$ does not change with the iteration and, in fact, if the time step and the mesh are fixed, one has the same matrix $A_h^{n, i} = A$ across all iterations, time steps, and realisations. Therefore, it is sufficient to assemble this matrix once, even before starting any other calculation, which significantly reduces the computational cost. This reduction is achieved at the expense of a linear convergence, as opposed to the quadratic convergence offered by Newton's method. However, the convergence is unconditional, in the sense that it does not depend on the initial guess, thus on the time step, and on the chosen spatial discretisation, or the mesh. Here we analyse the trade-off between the computational time saved by reducing the frequency of matrix computation and the slower linear convergence rate.

This manuscript is organised as follows. Section 2 introduces the assumptions and notations, followed by the formulation of the proposed gradient scheme. A regularised version of the gradient scheme for the stochastic Stefan problem is presented in Section 2.1. Section 3 focuses on the development of linearised solvers for both regularised and non-regularised gradient schemes. Specifically, the Newton method for the gradient scheme is described in Section 3.1, while the linearised solvers without and with regularisation are detailed in Sections 3.2 and 3.3, respectively, along with their convergence analyses. A comprehensive numerical study, including sensitivity analysis, is provided in Section 4 for both deterministic and stochastic cases of the Stefan problem.

\section{Problem formulation and Gradient scheme}

The Gradient discretisation method (GDM) will be used to approximate the solution to the above problem. The GDM, a generic framework, covers a large class of conforming and nonconforming schemes such as finite volume methods, Galerkin methods (including mass-lumped finite element), mixed finite element methods, hybrid high-order and virtual element methods, etc. The convergence results established for the GDM covers all the method within the framework. For further detail, refer to the monograph \cite{droniou2018gradient}. 

We recall here the main components of the gradient discretisation method.	
    \begin{definition}\label{def:GDF}
    Let $\mc{D}=(\xdo,\pd,\dd,\mc{I}_{\mc{D}},(t^{(n)})_{n=0,\cdots,N})$ be a space-time gradient discretisation (GD) for the homogeneous Dirichlet boundary conditions, with piecewise constant reconstruction, where:
		\begin{enumerate}
			\item[(i)] $\xdo$, the set of discrete unknowns, is a finite dimensional vector space on $\mb{R}$,
			\item[(ii)] $\pd:\xdo \rightarrow L^{\infty}(\Theta)$ is a linear piecewise constant reconstruction operator in the following sense: there exists a family of disjoint subsets $(\Theta_{i})_{i\in B}$ of $\Theta$ such that $\pd u=\sum_{i\in B}u_{i}\mathbbm{1}_{\Theta_{i}}$ for all $u= \sum_{i\in B}u_{i}\mv[e]_{i}\in\xdo$, where $B$ be the finite set of degree of freedom and $\mathbbm{1}_{\Theta_{i}}$ is the characteristic function of $\Theta_{i}$,
			\item[(iii)] The linear map $\dd: \xdo\rightarrow (L^{2}(\Theta))^{d}$ is the reconstructed discrete gradient such that $ \|\cdot\|_{\mc{D}}:= \|\dd \cdot\|_{L^{2}(\Theta)}$ is a norm on $\xdo$,
			\item[(iv)] 
            $\mc{I}_{\mc{D}}:Y\rightarrow\xdo$ is an interpolation operator which construct a discrete vector in the space of unknowns from an initial condition in $Y:=\{v\in L^{2}(\Theta): \zeta(v)\in H^1_0(\Theta)\}$,
			\item[(v)] $t^{(0)}=0<t^{(1)}<\cdots<t^{(N)}=T$ is a uniform time discretisation with a constant time step $\delta t_{\mc{D}}:=t^{(n+1)}-t^{(n)}$.
        \end{enumerate}
  
For each family $(v^{(n)})_{n=0,\cdots,N}\in\xdo$, we define piecewise-constant-in-time functions $\pd v:[0,T]\rightarrow L^{\infty}(\Theta)$, $\dd v:(0,T]\rightarrow (L^{2}(\Theta))^{d}$ and $\pd v^{(n)}(\bs{x}) \in L^\infty$ by:  
\[
\begin{aligned}
	&\pd v(0,\bs{x}):=\pd v^{(0)}(\bs{x}),\; \pd v(t,\bs{x}):=\pd v^{(n+1)}(\bs{x}), \\ &\dd v(t,\bs{x}):=\dd v^{(n+1)}(\bs{x}), \;\forall\; t\in(t^{(n)},t^{(n+1)}],\\ &  \text{for a.e.}\; \bs{x}\in \Theta,\; \forall\; n=0,\cdots,N-1.
\end{aligned}
\]
 
Moreover, the piecewise constant property of $\pd$ implies 
\begin{equation}\label{eq:pcro1}
\pd g(v)=g(\pd v)\quad \forall v\in \xdo,    
\end{equation}
where for any $g:\mb{R}\rightarrow\mb{R}$ such that $g(0)=0$ and $v=(v_i)_{i\in B}\in\xdo$, we set $g(v)=(g(v_i))_{i\in B}\in \xdo$. 
\end{definition}

We use standard notations and function spaces in the functional analysis: $L^2(\Theta)$, $L^\infty(\Theta)$, $H_0^1(\Theta)$, or its dual $H^{-1}(\Theta)$. $L^2(0, T; X)$ is the space of $X$-valued measurable functions that are square integrable in the sense of Bochner, where $X$ is  one of the spaces before.  $\lip \cdot,\cdot \rip$ denotes the inner product in $L^2(\Theta)$, or the duality pairing between $H^{-1}(\Theta)$ and $H^{1}_0(\Theta)$. $\| \cdot \|$ is  the norm in $L^2(\Theta)$, or the straightforward extension to $L^2(\Theta)^d$, and $\| \cdot \|_\infty$ the $L^\infty$ norm in $\Theta$ or in $(0, T] \times \Theta$. Where the meaning is  obvious, we replace $u(t,x)$ by $u$ or $u(t)$. $C$ denotes a generic positive constant independent of the discretisation parameters or the function itself. 

We will also use the following assumptions:
\begin{enumerate}[label=(A\arabic*)]
    \item\label{asum:zeta}  $\zeta: \mathbb R \rightarrow \mathbb R $ is increasing and Lipschitz continuous, with Lipschitz constant $L_\zeta > 0$. We also suppose that $\zeta(0) = 0$ and that $\zeta$ is coercive in the sense that there exists $c,d>0$ such that $|\zeta(s)|\ge c|s|-d$ for all $s, d\in\mathbb{R}$.
    \item\label{assum:W} Given the stochastic basis $(\Omega, \mathcal{F},\mathbb F=(\mathcal{ F}_{t})_{t\in [0,T]},\mathbb P)$, assume $W=\{W(t);t\in [0,T]\}$ is a Wiener process adapted to the filtration $\mb{F}$ taking values in a separable Hilbert space $\mathcal{K}(\Theta)$ with covariance operator $\mathcal{Q}$ such that $Tr(\mathcal{Q})<\infty$.
	\item\label{assum:f} $f:L^{2}\rightarrow \lkl$ is a continuous operator, where $\lkl$ is the Banach space of Hilbert-Schmidt operators with norm denoted as $\| \cdot \|_{\lkl}$. There exist $C_{1},\;C_{2}>0$ such that, for any $v\in L^{2}(\Theta)$ and $\Xi(v)\in L^{1}(\Theta)$, we assume
\begin{equation*}\label{eq:asum f}
\left\|f(\zeta (v)) \right\|^{2}_{\mathcal{L}(\mc{K}(\Theta),L^{2}(\Theta))}\le C_{2}\int_{\Theta}\Xi(v)dx+C_{1},
\end{equation*}
where $\Xi(v):=\int_{0}^{v}\zeta (s)ds$.
\item\label{assum:u0} $u_{0}:\Theta\rightarrow \mb{R}$ is a measurable function and $\Xi(u_0)\in L^1(\Theta)$.
\end{enumerate}

We define the discrete filtration $(\mc{F}^{n}_{N})_{0\le n\le N}$ by:
$$\mc{F}^{n}_{N}:=\sigma\{W(t^{k}):0\le k\le n\}\qquad\forall n=0,\ldots,N.$$

\begin{algo}[Gradient Scheme (GS) for \eqref{eq:sp}]\label{Algo}
Set $u^{(0)}:=\mc{I}_{\mc{D}}u_{0}$ and take random variables $u(\cdot)=(u^{(n)}(\omega,\cdot))_{n=0,\cdots,N}\in \xdo$ such that $u$ is adapted to the filtration $(\mc{F}^{n}_{N})_{0\le n\le N}$ and, for $n=1,\cdots, N$, for any $\varphi\in\xdo$ and for a.s. $\omega\in\Omega$, we have
\begin{equation}\label{eq:gs2}
\begin{aligned}
&\lip  \pd u^{n}(\omega), \pd\varphi\rip+\delta t_{\mc{D}}\lip \dd\zeta(u^{n})(\omega),\dd\varphi\rip\\&=\lip  \pd u^{n-1}(\omega), \pd\varphi\rip+\lip  f(\pd\zeta(u^{n-1}))(\omega)\Delta^{n}W(\omega),\pd\varphi\rip,
\end{aligned}
\end{equation}
where $\Delta^{(n)}W=W(t^{n})-W(t^{n-1})$.	
\end{algo}

Hereon, $\omega$ is omitted for legibility.

The convergence of this scheme is analysed in \cite{GDM_PorousMedium} for the deterministic case, and in \cite{DRONIOU2024114} for the stochastic case. For example, the error for the more regular variable $\zeta(u)$ is of order $\delta t + h$, see \cite[Corollary 2.1]{GDM_PorousMedium}. Here, we focus only on the analysis of numerical solvers used to solve the nonlinear system associated with this scheme. Throughout this article, $C$ represents an arbitrary constant which depends on $f,T,\mc{Q},u_{0},\Lambda$ and $\zeta$. Any further dependencies will be written in its subscript.  

\subsection{Regularisation-based approaches}

A popular technique when dealing with degenerate problem is to regularise the problem, i.e. to approximate $\zeta$ by a function $\zeta_\varepsilon$ having the derivative bounded away from 0 and from $\infty$ (in the fast diffusion case).
More precisely, with $\varepsilon > 0$ being a given parameter, we let $\zeta_\varepsilon : \R \to \R$ be a perturbation of $\zeta$ satisfying 
\begin{equation}
    \label{eq:zetaeps}
    \varepsilon \leq \zeta_\varepsilon^\prime \leq L_\zeta. 
\end{equation}
 A possible choice is 
\bea[eq:zetamax]
\zeta_\varepsilon(u) = \int_0^u \max\{\varepsilon, \zeta^\prime(s)\} \ d s. 
\eea
The resulting regularised equations have $\varepsilon$-dependent solutions $u_\varepsilon$ that approximate the one for the original, degenerate problem well, see \cite{Epperson1986}. 
Moreover, regularisation offers more possibilities to develop numerical algorithms. One is inspired by the fact that the unknown $u$ has typically a lower regularity than $\zeta(u)$, and therefore approximating the latter is more efficient. Once an approximation is found, since the regularised function $\zeta_\varepsilon$ is invertible, finding an approximation of $u$ is then straightforward.

Two different approaches are studied here in connection with regularisation. The first is to leave the regularised problem in terms of the less regular unknown $u_\varepsilon$. This makes sense since $u_\varepsilon$ has the same regularity as $\zeta_\varepsilon(u_\varepsilon)$. The drawback is that, in this case, the nonlinearity appears under the Laplacian. This can be seen immediately, if one replaces $\zeta$ by $\zeta_\varepsilon$ in \eqref{eq:gs2}. 

The second approach uses the fact that $\zeta_\varepsilon$ is invertible and rewrites \eqref{eq:gs2} in terms of the more regular variable $v_\varepsilon = \zeta_\varepsilon(u_\varepsilon$ and, obviously, $u_\varepsilon = \zeta^{-1}_\varepsilon(v_\varepsilon$. Note that, for the Stefan problem, this inverse does not exist since $\zeta$ remains constant on an interval. Therefore, for the second approach the regularisation is essential.  
Clearly, in this approach the Laplacian becomes linear, and the nonlinearity only appears as a zero-order term. All other terms, including the the nonlinear, second-order operator, are considered only on the known values from the preceding iteration or the initial guess, as follows from the formulation below.

Using $v_\varepsilon = \zeta_\varepsilon(u_\varepsilon)$ as the primary unknown, the regularised counterpart of \eqref{eq:sp} consists in finding the solution to
\begin{equation}\label{eq:risp}
\begin{cases}
&d(\zeta_\varepsilon^{-1}(v_\varepsilon))-\text{div}(\nabla v_\varepsilon) dt=f(v_\varepsilon) dW_{t}\quad \text{in}\quad \Theta_{T}:=(0,T)\times \Theta,\\&
v_\varepsilon(0,\cdot)=\zeta_\varepsilon(u_{0}) \quad \text{in}\quad \Theta,\\&
v_\varepsilon=0\quad \text{on}\quad (0,T)\times \partial\Theta. 
\end{cases}
\end{equation}
For this, we define the regularised GS, which can also be seen as a regularisation of Algorithm \ref{Algo}:

\begin{algo}[Regularised Gradient Scheme (RGS)]\label{algo:rigs}
Set $v_\varepsilon^{(0)}:=\mc{I}_{\mc{D}}\zeta_\varepsilon(u_{0})$ and take random variables $v_\varepsilon(\cdot)=(v_\varepsilon^{(n)}(\omega,\cdot))_{n=0,\cdots,N}\in \xdo$ such that $v$ is adapted to the filtration $(\mc{F}^{n}_{N})_{0\le n\le N}$ and, for any $\varphi\in\xdo$ and for almost every $\omega\in\Omega$, we have, for $n=1,\cdots, N$,
\begin{equation}\label{eq:rigs}
\begin{aligned}
&\lip  \pd \zeta_\varepsilon^{-1}(v_\varepsilon^{n}), \pd\varphi\rip+\delta t_{\mc{D}}\lip \dd v_\varepsilon^{n},\dd\varphi\rip\\&=\lip  \pd \zeta_\varepsilon^{-1}(v_\varepsilon^{n-1}), \pd\varphi\rip+\lip  f(\pd v_\varepsilon^{n-1})\Delta^{n}W,\pd\varphi\rip,
\end{aligned}
\end{equation}
where $\Delta^{n}W=W(t^{n})-W(t^{n-1})$.	
\end{algo}

The solution $v_\varepsilon$ in RGS plays the role of approximating $\zeta_\varepsilon(u_\varepsilon)$, and consequently $\zeta_\varepsilon^{-1}(v
_\varepsilon)$ is an approximation of $u_\varepsilon$. Moreover, as established in \cite[Theorem 5.1]{Epperson1986} and the subsequent discussion, it is known in the deterministic case that, for some constant $C$, the following inequality holds:
$$\|\zeta_\varepsilon(u_\varepsilon)-\zeta(u)\|_{L^2(\Theta_T)}\le C \varepsilon^{1/2}.$$
This implies, heuristically, that the RGS can introduce an error of about $\varepsilon^{1/2}$ to the approximation of the main GS \eqref{eq:gs2}. However this error can be controlled by an optimal choice of $\varepsilon$. To optimise it, we conducted a sensitivity analysis, detailed in the numerical experiments section. 

\section{Linear iterative schemes}

We observe that Algorithms \ref{Algo} and \ref{algo:rigs} require solving a sequence of nonlinear problems. To find an approximation to their solution, linear iterative schemes are needed. In the remaining of this work, different schemes are considered, some of them employing a regularisation step as discussed in the previous section. 

We start with the Newton method and then discuss two simple iterative schemes building on the L-Scheme mentioned before. The former is converging quadratically, but if the initial guess is close enough to the exact solution of the algebraic problem. Since, commonly, the initial guess is chosen as the solution computed at the previous time step, this induces a severe restriction on the time step, which also depends on the spatial discretisation and mesh, \cite{Radu_N}. Moreover, the discretisation matrix must be reassembled for every iteration, which can be quite time consuming. The latter schemes converge  linearly, but for mild restrictions on the time step, not depending on the spatial discretisation and mesh. Another advantage of the latter is that, at each time step, the operator on the left does not change with either the iteration, or with the realisation of the Wiener process. 
In the fully discrete case, if the spatial discretisation and time steps are not changed during the iterations, one solves a sequence of linear systems of the type $A y = b_i$, where the matrix $A$ remains the same, and only the vector on the right needs to be adjusted to $b_i$, $i \in \N_0$ being the iteration index. As a consequence, when using a direct linear solver, the matrix $A$ can be factorised once and for all (e.g., using a Choleski decomposition since $A$ is symmetric positive definite), and each nonlinear iteration is computationally extremely cheap.

In the rest of this section, we define the nonlinear solvers. As mentioned, the first is the classical Newton method, the other two are contraction-type methods, one for the original problem and the other for the regularised problem. In all cases, the initial guess at time $t_n$ consists in the solution obtained at time $t_{n-1}$.

\subsection{The Newton method}
Let $n \in \{1, \dots, N\}$ be fixed and let $i \in \N$ be the iteration index within a time step. In order to obtain the approximate solution for $u^n$ using the Newton method, consider $u^{n,i}=u^{n,i-1}+\delta u^{n,i}$ in \eqref{eq:gs2} and expand the Taylor series of $\zeta (u^{n,i})$ around $u^{n,i-1}$. %
The rigorous way of writing this iteration is through the following definition.

\begin{definition}[N solver]
    \label{def:newton}
    For $n\in \{0,\cdots, N\}$, the Newton solver (N) for \eqref{eq:gs2} consists in finding a sequence $(u^{n, i})_{i\in \NN}$ such that, for all $i\in\NN$, $u^{n,i}\in\xdo$ satisfies
    \begin{equation}\label{weak.sol.Newton} 
        \begin{split}
        \lip \delta (\pd u^{n,i}),\pd\varphi\rip &+  \delta t_{\mc{D}}\lip\dd\Big(   \zeta'( u^{n, i-1})\; \delta u^{n,i}\Big), \dd \varphi\rip\\ &=\lip \pd u^{n-1},\pd \varphi\rip + \lip  f(\zeta(\pd u^{n-1}))\Delta^{n}W, \pd \varphi\rip\\& - \lip \pd u^{n,i-1},\pd \varphi\rip -\delta t_{\mc{D}}\lip \dd \zeta( u^{n, i-1}), \dd \varphi\rip.
        \end{split}
    \end{equation}
    for all $\varphi \in \xdo$.
\end{definition}%

Since $\zeta' (u^{n,i-1})$ must be evaluated in every iteration, the Newton method requires assembling the complete linear system for each iteration in a fully discrete case.

\subsection{The linear iterative scheme without regularisation}
To approximate the solution of the GS defined in Algorithm \ref{eq:gs2} for the non-regularised problem,  we consider the constant $L \geq L_\zeta/2$, and define the linearised solver for \eqref{eq:gs2} as follows:
\begin{definition}[L solver]
    \label{def:lin_non_reg}
    For $n\in \{1,\cdots, N\}$, the linearised solver (L) for \eqref{eq:gs2} is a sequence $(u^{n,i})_{i\in\NN}$ such that, for all $i\in\NN$, $u^{n,i}\in\xdo$ satisfies
    \begin{equation}\label{eq:lin_non_reg} 
\begin{split}
     \lip  \pd u^{n, i},\pd \varphi\rip + & \delta t_{\mc{D}}L  \lip  \dd u^{n, i}, \dd \varphi\rip \\
     ={}& \delta t_{\mc{D}} \lip  L \dd u^{n, i-1} - \dd \zeta(u^{n, i-1}), \dd \varphi\rip \\
     &+  \lip \pd u^{n-1}, \pd\varphi\rip + \lip  f(\zeta(\pd u^{n-1}))\Delta^{n}W, \pd\varphi\rip.  
\end{split}
\end{equation}
for all $\varphi \in \xdo$.
\end{definition}

This solver is similar to the one discussed in 
\cite{GDM_PorousMedium}, where the problem is formulated as a coupled system, in terms of the pair $(u^{n, i}, w^{n, i})$ with the unknown $w^{n, i}$ given by $w^{n, i} = L \big(u^{n, i} - u^{n, i-1}\big)  + \zeta(u^{n, i-1})$. In \eqref{eq:lin_non_reg} we have eliminated $w^{n,i}$ by simple substitution. 

In the following lemma, we show the convergence of sequence of iterations $\{u^{n, i}, i \in \N\}$ (in \eqref{eq:lin_non_reg}) to $u^{n}$ (in \eqref{eq:gs2})
in a dual norm $\|\cdot\|_{*,\mc{D}}$ on $\pd (\xdo)\subset L^{2}(\theta)$ which is defined as: for all $v\in \pd(\xdo)$
$$
\|v\|_{*,\mc{D}}:=\sup\left\{ \int_{\Theta} v(x)\pd w(x) dx\;: w\in \xdo,\;\|\dd w\|=1\right\}.
$$
This convergence holds for any initial guess, but a natural choice is $u^{n, 0} = u^{n-1}$.

\begin{lemma}\label{lem:conv_unreg}
Let $L \geq L_\zeta/2$, $n \in \{1, \dots, N\}$ and assume that $u^{n-1} \in \xdo$ is known. Further, let $u^{n}$ be a weak solution to the GS \ (Algorithm \ref{Algo}) and consider the linearised solver \eqref{eq:lin_non_reg}, with an arbitrary initial guess $u^{n, 0} \in \xdo$. Then, the sequence $\{u^{n, i}, i \in \N\}$ converges in the dual norm $\|{\cdot}\|_{*,\mc{D}}$ to $u^{n}$. 
\end{lemma}
\begin{proof}
Let the error
\begin{equation*}
    \label{eq:errorsu}
    \eui := u^n - u^{n, i}  
\end{equation*}
and subtract \eqref{eq:lin_non_reg} from \eqref{eq:gs2} to obtain, for any $\varphi \in \xdo$,
\begin{equation}
    \label{eq:Lem_l_1}
\begin{split}
    \lip  \pd \eui,\varphi\rip + & \delta t_{\mc{D}}L \lip \dd \eui, \dd \varphi\rip \\
    & = \delta t\lip  L \dd \eum - \dd \big(\zeta(u^{n}) - \zeta(u^{n, i-1})\big), \dd \varphi\rip. 
\end{split}
\end{equation}

We will use as test function in \eqref{eq:Lem_l_1} the discrete solution to the Poisson problem 
$$
\left\{
\begin{array}{rcll}
- \Delta G &=& \eui, \quad &\text{ in } \Theta\\
G &=& 0 &\text{ at } \partial \Theta,
\end{array}
\right.
$$
as given by \cite[Definition 4.15]{droniou2018gradient}: $G^{n,i}\in\xdo$ is the solution of
\begin{equation}
    \label{eq:Lem_l_2}
    \lip \dd G^{n, i}, \dd \psi\rip = \lip \pd\eui, \pd \psi\rip\qquad\forall \psi \in \xdo.
\end{equation}
Taking $\varphi = G^{n, i}$ in \eqref{eq:Lem_l_1} we obtain 
\[\begin{split}
    \lip \pd \eui, \pd G^{n, i}\rip + & \delta t_{\mc{D}}L \lip\dd \eui, \dd G^{n, i}\rip \\
    & = \delta t_{\mc{D}}\lip L \dd \eum - \dd \big(\zeta(u^{n}) - \zeta(u^{n, i-1})\big), \dd G^{n, i}\rip.  
\end{split}
\]
Using $\eui$, $\eum$ and $\zeta(u^{n}) - \zeta(u^{n, i-1})$ as test functions in \eqref{eq:Lem_l_2}, and the equality $\| \dd G^{n, i}\| = \|\pd\eui\|_{*,\mc{D}}$  (see \cite[Eq.~(4.29)]{droniou2018gradient}), together with \eqref{eq:pcro1}, one obtains  
\begin{multline*}
    \|\pd \eui\|_{*,\mc{D}}^2 + \delta t_{\mc{D}}L \|\pd \eui\|^2 \\
    = 
    \delta t_{\mc{D}}\lip L\;\pd \eum - \big(\zeta(\pd u^{n}) - \zeta(\pd u^{n, i-1})\big),  \pd\eui\rip.  
\end{multline*}

Using Lemma \ref{lem:zetaLab} (see Appendix) with $a=u^n$ and $b=u^{n,i}$, and the H\"{o}lder inequality, this implies
\begin{equation}
    \|\pd \eui\|_{*,\mc{D}}^2 + \delta t_{\mc{D}}L \|\pd \eui\|^2\le \delta t_{\mc{D}}L\lip |\pd \eum| ,  |\pd\eui|\rip  
\end{equation}
The Cauchy-Schwarz and Young inequalities then implies
$$
    \|\pd \eui\|_{*,\mc{D}}^2 + \delta t_{\mc{D}}L \|\pd \eui\|^2 \leq  
    \cfrac {\delta t_{\mc{D}}L} 2 \|\pd \eui\|^2  + \cfrac {\delta t_{\mc{D}}L} {2}\|\pd \eum\|^2,    
$$
or, after multiplication by 2 and reducing the terms in $\eui$, 
\begin{equation}
    \label{eq:7}
    2 \|\pd \eui\|_{*,\mc{D}}^2 + \delta t_{\mc{D}}L \|\pd \eui\|^2 \leq \delta t_{\mc{D}}L \|\pd \eum\|^2.  
\end{equation}
One can add the inequality in \eqref{eq:7} for $i = 1, 2, \dots$ to obtain, by telescopic sum,
\begin{equation}
    \label{eq:8}
    2 \sum_{i = 1}^\infty \|\eui\|_{*,\mc{D}}^2 \leq \delta t_{\mc{D}}L \|e_u^{n, 0}\|^2 < \infty,   
\end{equation}
implying that the series on the left is absolutely convergent. This gives the desired result.
\end{proof}

\begin{remark}\label{rem:convL}
    The convergence result in Lemma \ref{lem:conv_unreg} imposes no restriction on the time step $\delta t_{\mc{D}}$. It can be extended straightforwardly to any spatial discretisation and mesh. However, the convergence is given only in the dual norm, and without providing a convergence order. 
    When compared to other convergence results, as obtained e.g. for regularisation-based iteration schemes for Algorithm \ref{algo:rigs}, the convergence result in Lemma \ref{lem:conv_unreg} is weaker. A possible explanation for this is that, for such type of problems, $u$ has a low regularity and may even become discontinuous. The linear scheme in \eqref{eq:lin_non_reg} involves the discrete gradients $\dd u^{n, i-1}$ and $\dd u^{n, i}$, which are an approximation of $\nabla u$ at the time step $n$. Around discontinuities, $\dd u^{n, i}$ becomes steeper when refining the mesh, which reduces the speed of convergence. This effect may be reduced in the Newton iteration, as $L$ under $\dd$ is replaced by $\zeta'(u^{n, i-1})$ and, if the method converges, this vanishes for the values where the discontinuity may occur. To deal with such issues, in Definition \ref{def:srniwl} a simple iterative method is discussed that uses the regularised $\rzeta$, but not its inverse.
\end{remark}

\subsection{The linear iterative scheme for the regularised problem}\label{sec:reg}

For some $L>0$, we now introduce a regularised-linearised solver, with a regularisation as defined in \eqref{eq:zetaeps}, to approximate the regularised GS (Algorithm \ref{algo:rigs}). The constant $L$ only depends upon the regularisation parameter $\varepsilon$ and is specified in the next lemma. 

\begin{definition}
    \label{def:weak.ireg}
Let $\zeta_\varepsilon$ be a regularisation of $\zeta$ satisfying \eqref{eq:zetaeps}. For $n\in \{1,\cdots, N\}$, an iterative solver for \eqref{eq:rigs} is a sequence  $(v_\varepsilon^{n,i})_{i\in\NN}$ such that, for all $i\in\NN$, $v_\varepsilon^{n,i}\in\xdo$ satisfies
\begin{equation}\label{eq:reg_is} 
\begin{split}
     L \lip  \pd v_\varepsilon^{n, i},\pd \varphi\rip + & \delta t_{\mc{D}}   \lip  \dd v_\varepsilon^{n, i}, \dd \varphi\rip \\
     ={}& \lip L\; \pd v_\varepsilon^{n, i-1} - \pd \zeta_\varepsilon^{-1}(v_\varepsilon^{n, i-1}),\pd \varphi\rip \\&+  \lip \pd \zeta_\varepsilon^{-1}(v_\varepsilon^{n-1}) , \pd\varphi\rip + \lip  f(\pd v_\varepsilon^{n-1})\Delta^{n}W, \pd\varphi\rip.  
\end{split}
\end{equation}
for all $\varphi \in \xdo$.
\end{definition}%
The initial guess $v_\varepsilon^{n, 0}$ can be chosen arbitrarily, but a reasonable choice is $v_\varepsilon^{n, 0} = v_\varepsilon^{n-1}$. 

Below we use the discrete Poincar{\' e} inequality,  
\begin{equation}\label{eq:poincare2}
\| \pd v\| \leq C_{\mc{D}} \|\dd v \|, 
\end{equation}
for all $v\in \xdo$, where $C_\mc{D}$ is a constant that only depends on discretisation $\mc{D}$ (refer to \cite[Remark 2.3]{droniou2018gradient} for further detail), but not on $v$. We introduce the following notations: for the error,
\begin{equation}
    \label{eq:errorsv}
    \evi = v_\varepsilon^n - v_\varepsilon^{n, i}, 
\end{equation}
and for the norm
\begin{equation}
\label{eq:normX}    
  \|v\|_{X}^2:=\Big(L+\cfrac{\delta t_{\mc{D}}}{C_{\mc{D}}^2}\Big)\|\pd v\|^2+ \delta t_{\mc{D}} \|\dd v \|^2\qquad \forall v\in \xdo.
\end{equation}

\begin{lemma}\label{lem:conv_reg}
Let $L \ge \frac{1}{\varepsilon}$ and $n \in \{1, \dots, N\}$ be fixed, and assume that $v_\varepsilon^{n-1} \in \xdo$ is given. Further, let $v_\varepsilon^n$ be a weak solution to the RGS (Algorithm \ref{algo:rigs}) and consider the iteration in \eqref{eq:reg_is}, with an arbitrary initial guess $v_\varepsilon^{n, 0} \in \xdo$. Then, the sequence $\{v_\varepsilon^{n, i}, i \in \N\}$ converges in the $\| \cdot \|_X$ norm to $v_\varepsilon^n$. Moreover, one has, for all $i$,
$$
\| \evi \|_{X} \leq \alpha \| \evm \|_{X}, \quad \text{ where } \quad \alpha = \left(L - L_\zeta^{-1}\right) \left[L \left(L + \cfrac{\delta t}{C_{\mc{D}}^2}\right)\right]^{-\frac 1 2} < 1.
$$
\end{lemma}
\begin{proof}
Subtracting \eqref{eq:reg_is} from \eqref{eq:rigs} and plugging $\evi= v_\varepsilon^n - v_\varepsilon^{n, i}$ to get
\begin{equation}\label{eq:1lm2}
    \begin{aligned}
        & L \lip \pd \evi, \pd \varphi\rip + \delta t_{\mc{D}} \lip \dd \evi, \dd \varphi \rip
        \\
        & \hskip 2em = L \lip \pd e_{v_\varepsilon}^{n,i-1},\pd \varphi\rip 
        -\lip \rzeta^{-1}(\pd v_\varepsilon^{n}) - \rzeta^{-1}(\pd v_\varepsilon^{n,i-1}),\pd \varphi\rip. 
    \end{aligned}
\end{equation}       
Choosing $\varphi=\evi$ and taking upper bound using absolute value, we deduce
\begin{equation}\label{eq:1lm3}
    \begin{aligned}
        & L\|\pd \evi\|^2 + \delta t_{\mc{D}} \|\dd \evi\|^2 \\
        & \hskip 2em \le 
        \lip \left|L (\pd e_{v_\varepsilon}^{n,i-1}) -\left( \rzeta^{-1}(\pd v_\varepsilon^{n}) - \rzeta^{-1}(\pd v_\varepsilon^{n,i-1})\right)\right|,\left|\pd \evi\right|\rip. 
    \end{aligned}
\end{equation}  
Since $L\ge\varepsilon^{-1}$ and, by \eqref{eq:zetaeps}, $(L_\zeta)^{-1} \le (\rzeta^{-1})'\le \varepsilon^{-1}$, using Lemma~\ref{lem:ZLab} (see Appendix) with $Z=\rzeta^{-1}$, $L_Z=(L_\zeta)^{-1}$ and $a-b=\pd \evi$, we obtain
\begin{equation*}\label{eq:1lm31}
    \begin{aligned}
        L \lip \pd \evi,\pd \varphi\rip +\delta t_{\mc{D}} \lip \dd \evi, \dd \varphi \rip
        &\le (L-L_\zeta^{-1})\left|\lip |\pd \evm|,|\evi|\rip \right|. 
    \end{aligned}
\end{equation*} 
Applying the H\"{o}lder inequality, we get
\begin{equation*}
        L\|\pd \evi\|^2 +\delta t_{\mc{D}} \|\dd \evi\|^2 \\
        \le  (L-L_\zeta^{-1})\| \pd e_{v_\varepsilon}^{n,i-1}\| \|\pd \evi\|.
\end{equation*}%
Using the Young inequality and after multiplication by 2 one obtains 
$$
   L\|\pd \evi\|^2+ 2\delta t_{\mc{D}} \|\dd \evi \|^2 \le \cfrac{(L-L_\zeta^{-1}) ^2}{L} \|\pd e_{v_\varepsilon}^{i-1}\|^2.
$$
Finally, using the discrete Poincar\'e inequality, we infer
\begin{multline*}
       \Big(L+\cfrac{\delta t_{\mc{D}}}{C_{\mc{D}}^2}\Big)\|\pd \evi\|^2+ \delta t_{\mc{D}} \|\dd \evi \|^2  \\
       \le \cfrac{\Big(L-L_\zeta^{-1}\Big)^2}{L\Big(L+\cfrac{\delta t_{\mc{D}}}{C_{\mc{D}}^2}\Big)} \Bigg[\Big(L+\cfrac{\delta t_{\mc{D}}}{C_{\mc{D}}^2}\Big)\|\pd e_{v_\varepsilon}^{i-1}\|^2+\delta t_{\mc{D}} \|\dd \evi \|^2\Bigg],
\end{multline*}
where the term involving $\|\dd \evi \|$ in the right-hand side has been added by non-negativity. Recalling the notations introduced in \eqref{eq:errorsv} and \eqref{eq:normX}, this can be rewritten as
$$
\| \evi\|_{X}\le \alpha \| \evm\|_{X},    
$$
with the contraction constant $\alpha$ defined by 
\begin{equation}\label{eq:alpha}
    \alpha =(L-L_\zeta^{-1})\left[L\Big(L+\cfrac{\delta t_{\mc{D}}}{C_{\mc{D}}^2}\Big)\right]^{-1/2} < 1.
\end{equation}
\end{proof}
\begin{remark}\label{remark:reg_plateau}
Clearly, the larger $L$, the closer the convergence rate, $\alpha$, is to 1. Moreover, as $L>1/\varepsilon$, an intention to reduce the regularisation error force to choose a larger $L$ and eventually slower convergence. One chooses $L$ as small as possible, to ensure a rapid convergence. %
Optimistically, the value of $\alpha$ can not be better than the value for $L =  1 /\varepsilon$, that is
\begin{align*}
\alpha &= \left( 1 - \cfrac \varepsilon  {L_\zeta}\right) 
\left(1  + \cfrac {\varepsilon \delta t_{\mc{D}}} {C_\Omega^2}\right)^{-\frac 1 2}
.  
\end{align*}
From the above, one observes that, the smaller $\delta t_{\mc{D}}$ and $\varepsilon$ are, the closer the convergence rate is to 1. 
{In the deterministic setting, in \cite{Epperson1986} it is shown that the regularisation adds an $L^2$ error of order $\varepsilon^{1/2}$ in terms of $u$, respectively of order $\varepsilon$ in terms of $\zeta(u)$. Therefore, to ensure that the regularisation error does not dominate the error induced by the discretisation in time, when choosing $\varepsilon$ one should take $\delta t_{\mc{D}}$ into account.} 
This implies that the regularised solver \eqref{eq:reg_is} becomes computationally expensive w.r.t. the number of iterations, as $\delta t_{\mc{D}}$ approaches zero. 
This slow convergence is due to \emph{inverting} the regularised nonlinearity $\rzeta$. For the Stefan problem, as $\zeta$ is constant for $u\in[0,1]$, the regularisation defined in \eqref{eq:zetamax} yields $v=\zeta_\varepsilon(u)=\varepsilon u$ in this interval, so $L \ge \varepsilon^{-1}$.
\end{remark}

As follows from Remark \ref{rem:convL}, employing   regularisation is justified by stronger convergence results or improved stability of the schemes compared to those for algorithms without regularisation. On the other hand, Remark \ref{remark:reg_plateau} recommends avoiding the use of $\rzeta^{-1}$. Letting again the regularisation $\rzeta$ satisfy \eqref{eq:zetaeps}, as for Algorithm \ref{algo:rigs} we consider the following iterative scheme.
\begin{definition}[R solver]
    \label{def:srniwl}
    Let $\zeta_\varepsilon$ be a regularisation of $\zeta$ satisfying \eqref{eq:zetaeps}. For $n\in \{1,\cdots, N\}$, the linearised regularised solver (R) for \eqref{eq:gs2} is a sequence $(u_\varepsilon^{n, i})_{i\in\NN}$ such that, for all $i\in\NN$, $u_\varepsilon^{n,i}\in\xdo$ satisfies
\begin{equation}\label{eq:rlwithoutInv} 
\begin{split}
     \lip  \pd u_\varepsilon^{n, i},\pd \varphi\rip + & \delta t_{\mc{D}}L  \lip  \dd u_\varepsilon^{n, i}, \dd \varphi\rip \\
   ={}& \delta t_{\mc{D}} \lip  L \dd u_\varepsilon^{n, i-1} - \dd \zeta_\varepsilon(u_\varepsilon^{n, i-1}), \dd \varphi\rip \\&+  \lip \pd u_\varepsilon^{n-1}, \pd\varphi\rip + \lip  f(\zeta_\varepsilon(\pd u_\varepsilon^{n-1}))\Delta^{n}W, \pd\varphi\rip,  
\end{split}
\end{equation}
for all $\varphi \in \xdo$.
\end{definition}
The convergence of the sequence $(u_\varepsilon^{n,i})_{i\in\NN}$ towards $\zeta^{-1}_\varepsilon(v_\varepsilon^n)$, where $v_\varepsilon^{n}$ solves \eqref{eq:rigs}, can be obtained following the steps in the proof of Lemma \ref{lem:conv_unreg}.

\section{Numerical experiments}
In this section, we compare the computational efficiency of the Newton method (N), the linearised solver (L) and the regularised solver (R), as defined in \ref{def:newton}, \ref{def:lin_non_reg} and \ref{def:srniwl} respectively. The comparison is made in terms of their computation time for both the deterministic and stochastic cases of the Stefan problem. To ensure a fair comparison, each simulation was repeated at least 10 times, and the minimum CPU time among the repetitions was recorded. These repetitions compare the intrinsic efficiency of each numerical method without being influenced by external system overheads (e.g., background processes, random system delays), which can vary across runs and are not related to the method itself. Since such overheads can artificially inflate the CPU time, taking the average over several runs would still retain some of this randomness and may not reflect the true computational cost of the algorithm. Instead, by repeating each test multiple times and reporting the minimum CPU time, we aim to approximate the best-case scenario -- that is, the time closest to what the algorithm would achieve in the absence of external disturbances. This approach is commonly used when the goal is to evaluate the pure performance characteristics of a numerical method.

\subsection{The deterministic case}

We begin with the deterministic case, considering a non-homogeneous Dirichlet boundary condition where $\Theta=(0,1)^{2}$, $T=1$,
\begin{equation*}
	\zeta(u) =	
	\begin{cases}
		u, & \text{if } u \leq 0\\
		1, & \text{if } 0\le u\leq 1\\
		u-1, & \text{if } 1\le u
	\end{cases}	
\end{equation*}
and $f=0$, for which there exist an analytical solution
\begin{equation*}
	u(x_{1},x_{2},t) =	
	\begin{cases}
		2\exp (t-x_{1})-1\quad (>1), & \text{if } x_{1} < t\\
		\exp (t-x_{1})-1\quad (<0), & \text{if } t<x_{1}.		
	\end{cases}	
\end{equation*}
The following results are obtained using the mass-lumped $\mathcal P^1$ (MLP1) finite elements and the code which is available at \url{https://github.com/awaismaths/LR_Stefan}. The numerical computations are performed for a uniform time step, for which we use a simplified notation $\delta t$ instead of $\delta t_{\mc{D}}$. The key distinctions between each approach are as follows:
\begin{itemize}
    \item Newton: In this method, the linearised system (e.g., $A u=b$)  is re-assembled at each iteration across all the time steps, requiring the decomposition of the coefficient matrix (e.g., through Cholesky decomposition, $A=LL^T$) at every step. Although this incurs a higher computational cost per iteration, Newton's method benefits from super-linear convergence. This means, fewer iterations are needed to achieve convergence. However, this convergence is guaranteed under restrictions on $\delta t$ with respect to the mesh $h$.
    \item Linearised: The coefficient matrix ($A$) remains unchanged throughout all iterations for all time steps. This allows for significant speedup by pre-computing the decomposition (e.g. $A=LL^t$) just once at the beginning. However, more iterations are generally required due to linear convergence, but for without any restriction on $\delta t$ or $h$.
    \item Regularised: Similar to the linearised solver, the coefficient matrix remains constant across iterations and time steps. However, the regularisation parameter can be adjusted to enhance the efficiency. This adjustment is particular beneficial in terms of efficiency when the solution takes values in regions where $\zeta$ remains constant, see Remark \ref{remark:reg_plateau}.
\end{itemize}
\begin{table}
										\centering
										\caption{Data for the triangular meshes}%
										\label{tab:tri_mesh}
										\begin{tabular}{|c|c|c|c|c|}
												\hline
												Mesh& Size& Cells & Edges&Vertices\\
												\hline
												M1&    0.250&   56&     92&     37\\
												M2&    0.125&   224&    352&    129\\
												M3&    0.0625&   896&    1376&   481\\
												M4&    0.0313&   3584&   5440&   1857\\
												M5&    0.0156&   14336&   21632&   7297\\
												M6&    0.0078&   57344&  86272&  28929\\
												\hline	
											\end{tabular}%
	\end{table}%
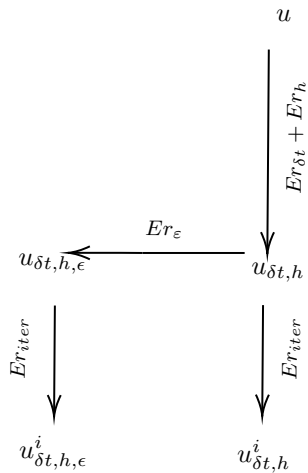
\begin{figure}[!ht]
\centering
\tikzset{every picture/.style={line width=0.75pt}} 
\begin{tikzpicture}[x=0.75pt,y=0.75pt,yscale=-1,xscale=1]

\draw    (309.92,148.1) -- (361.43,148.1) ;
\draw    (309.92,148.1) -- (298.88,148.03) -- (274.43,148.03) ;
\draw [shift={(272.43,148.03)}, rotate = 360] [color={rgb, 255:red, 0; green, 0; blue, 0 }  ][line width=0.75]    (10.93,-3.29) .. controls (6.95,-1.4) and (3.31,-0.3) .. (0,0) .. controls (3.31,0.3) and (6.95,1.4) .. (10.93,3.29)   ;
\draw    (373.62,45.42) -- (373.34,84.7) -- (372.89,146.55) ;
\draw [shift={(372.88,148.55)}, rotate = 270.41] [color={rgb, 255:red, 0; green, 0; blue, 0 }  ][line width=0.75]    (10.93,-3.29) .. controls (6.95,-1.4) and (3.31,-0.3) .. (0,0) .. controls (3.31,0.3) and (6.95,1.4) .. (10.93,3.29)   ;
\draw    (265.72,174.63) -- (265.33,228.77) ;
\draw [shift={(265.32,230.77)}, rotate = 270.41] [color={rgb, 255:red, 0; green, 0; blue, 0 }  ][line width=0.75]    (10.93,-3.29) .. controls (6.95,-1.4) and (3.31,-0.3) .. (0,0) .. controls (3.31,0.3) and (6.95,1.4) .. (10.93,3.29)   ;
\draw    (370.71,174.42) -- (370.32,228.55) ;
\draw [shift={(370.31,230.55)}, rotate = 270.41] [color={rgb, 255:red, 0; green, 0; blue, 0 }  ][line width=0.75]    (10.93,-3.29) .. controls (6.95,-1.4) and (3.31,-0.3) .. (0,0) .. controls (3.31,0.3) and (6.95,1.4) .. (10.93,3.29)   ;

\draw (375.81,23.84) node [anchor=north west][inner sep=0.75pt]  [rotate=-0.41]  {$u$};
\draw (309.83,130.23) node [anchor=north west][inner sep=0.75pt]  [font=\footnotesize,rotate=-0.41]  {$Er_{\varepsilon }$};
\draw (381.03,119.97) node [anchor=north west][inner sep=0.75pt]  [font=\footnotesize,rotate=-270.41]  {$Er_{\delta t} +Er_{h}$};
\draw (363.4,151.52) node [anchor=north west][inner sep=0.75pt]  [rotate=-0.41]  {$u_{\delta t,h}$};
\draw (245.87,147.03) node [anchor=north west][inner sep=0.75pt]  [rotate=-0.41]  {$u_{\delta t,h,\epsilon }$};
\draw (241.83,213.38) node [anchor=north west][inner sep=0.75pt]  [font=\footnotesize,rotate=-270.41]  {$Er_{iter}$};
\draw (378.84,213.53) node [anchor=north west][inner sep=0.75pt]  [font=\footnotesize,rotate=-270.41]  {$Er_{iter}$};
\draw (246.29,239.05) node [anchor=north west][inner sep=0.75pt]  [rotate=-0.41]  {$u_{\delta t,h,\epsilon }^{i}$};
\draw (356.28,240.84) node [anchor=north west][inner sep=0.75pt]  [rotate=-0.41]  {$u_{\delta t,h}^{i}$};

\end{tikzpicture}
\caption{Types of errors}
\label{fig:approx_figure}
\end{figure}
The accuracy and efficiency of these solvers are influenced by three primary types of errors: time discretisation error ($Er_{\delta t}$), space discretisation error ($Er_{h}$), and nonlinear system approximation error ($Er_{it}$), which is measured by the $L^2$ norm of the residual of the nonlinear system. In the regularised method, an additional error ($Er_\varepsilon$) arises from the regularisation parameter. These errors are presented in Figure \ref{fig:approx_figure}. For an efficient approximation of the fully-discrete solution, all these errors should be in balance. More precisely, the time step, the mesh, the regularisation parameter and the tolerance parameter in the iterative solver should be chosen in such a way that the corresponding errors have the same order. For example, taking a very low tolerance to reduce drastically the error $Er_{it}$ will not impact significantly the overall error if $\delta t$ or $h$ and, if applicable, $\varepsilon$ remain large. 
{In this context, we mention that, if converging and due to its quadratic convergence, the Newton method will provide within few iterations a solution with a very small $Er_{it}$, while many more iterations are requested for the L and R methods. However, the overall error may still remain large, as it is dominated by the other error components, $Er_{\delta t}$ and $Er_{h}$.}
Hence, to minimise the overall error 
{and for the sake of efficiency}, each 
error 
{component} must be finely tuned, implying that the tolerance should depend on both space and time discretisation, and on $\varepsilon$ in case of the regularised method.

To gain deeper insights, we conducted a sensitivity analysis of these methods. In this analysis, we considered the following error norms calculated at the final time $T$, where the indices $e$ and $a$ represent the exact solution of the deterministic Stefan problem and approximate solution up to the tolerance defined below, respectively:
\begin{align}
    E_{\zeta} & = \frac{\|\zeta(u^T_e) - \zeta(u^T_a)\|_{L^2(\Theta)}}{\|\zeta(u^T_e)\|_{L^2(\Theta)}} \quad \text{and} \\
    E_{\nabla\zeta} &= \frac{\|\nabla\zeta(u^T_e) - \nabla\zeta(u^T_a)\|_{L^2(\Theta)}}{\|\nabla\zeta(u^T_e)\|_{L^2(\Theta)}}.
    \end{align}
To determine an appropriate tolerance level, we introduce parameters $C_{\text{tol}} = \{ 1, 10 ,10^2, 10^3, 10^4 \}$ and $C_\varepsilon=\{1, 10^{1}, 10^{2}, 10^{3} \}$ and define the tolerance and the regularisation parameters as 
\begin{equation}\label{eq:C_Param}
tol = \frac{\min(\delta t^2, h^2)}{C_{\text{tol}}}, \quad \text{ and } \quad \varepsilon = \frac{\min(\delta t, h)}{C_\varepsilon}. 
\end{equation}
By the sensitivity analysis using these parameters, we intend to get numerically optimal values of $\text{tol}$ and $\varepsilon$.  This approach minimises unnecessary iterations, thereby optimising the computational time while maintaining the desired accuracy. In the first experiment, we fixed 
$C_\varepsilon = 1$ and varied $C_{\text{tol}}$. Figures \ref{fig:tol_mesh1_3} and \ref{fig:tol_mesh4_6} present a series of plots for various choices of $\delta t$ (decreasing from left to right) and $h$ (decreasing from top to bottom). Each plot depicts the errors $E_{\zeta}$ and $E_{\nabla\zeta}$ (on the y-axis) for fixed values of $\delta t$ and $h$, while varying $C_{\text{tol}}$ (on the x-axis).
\input{chapters/fig/sam1_3}%
\input{chapters/fig/sam4_6}%

 In these figures, immediately noticeable is the fact that, for $C_{\text{tol}}\ge 100$, all of these methods shows almost the same relative errors. This means that $C_{\text{tol}}=100$ is an optimal choice that works for each solver without compromising on accuracy. We further notice that $C_{\text{tol}}=10$ is also a good choice for $\delta t=0.1$ which works for all mesh sizes on a negligible cost. For $\delta t=0.01$, there is a noticeable difference in the error, which is more prominent in Figure \ref{fig:tol_mesh4_6}. However, the relative difference of the error with respect to the larger one is less than a unit, i.e., bounded by an order of magnitude. 
 Therefore, $C_{\text{tol}}=10$ is also a reasonable choice to test the efficiency, whereas choosing one of them ($C_{\text{tol}}=10$or $100$) depends on whether one prefers to trade off accuracy for efficiency or vice versa. 

One more thing to note is that the error plots for the Newton are flat in most of the cases. This is because, in these cases, Newton has already attained the saturation, i.e., further iterations will not improve the approximation. This is due to the fact that, when the tolerance is not too large, because of the 
quadratic convergence Newton goes from just above the tolerance to machine precision in only one iteration. See Table \ref{tab:residue_comp} where "$res$" represents the $L^2$ norm of the residue of the non-linear system. We see that the residue by the Newton method is saturated at $C_{\text{tol}}\ge 100$ while the rest of methods are moving along with the tolerance. %
\begin{table}[h!]
\centering
\begin{tabular}{|l|l|l|l|l|}
\hline
$C_{tol}$    & $tol$  & res\_N   & res\_L   & res\_R   \\ \hline
1    & 1.25E-01 & 1.61E-03 & 7.11E-03 & 6.91E-03 \\ \hline
10   & 1.25E-02 & 1.16E-04 & 7.59E-04 & 7.44E-04 \\ \hline
100  & 1.25E-03 & 4.04E-17 & 7.99E-05 & 7.91E-05 \\ \hline
1000 & 1.25E-04 & 4.04E-17 & 8.17E-06 & 7.99E-06 \\ \hline
\end{tabular}
\caption{Residue of the nonlinear systems, $h=0.125$ and $dt=0.01$}
\label{tab:residue_comp}
\end{table}%

In the next experiment, we fixed $C_{\text{tol}}=100$ and varied $C_{\varepsilon} = \{1, 10, 10^2, 10^3 \}$. The results, as shown in Figure \ref{fig:ep_mesh1_5}, indicate that the errors for $C_{\varepsilon}=10$ are almost identical across different mesh sizes and time steps, with plots virtually overlapping. This implies that $C_{\text{tol}}=100$ and $C_{\varepsilon}=10$ can be chosen if one is extremely cautious about the error. %
\begin{figure}\centering  
      \begin{minipage}{0.5\textwidth}
      \ref{epmesh1to3}
      \vspace{0.10cm}
      \end{minipage}  
    \begin{tikzpicture}[scale=0.7]
      \begin{loglogaxis}[
      ymin=0.001, ymax=0.3,
      name=plot1, ytick pos=left,
      legend columns=6, legend to name=epmesh1to3,
      ylabel={$h=0.25$},xlabel={$\delta t=0.1$}
      ]
        \addplot [mark=asterisk, mark options={solid,color=cyan}, cyan,dashed] table[x=C,y=L2z_N] {ep/eph.0625dt0.1.dat};
        \addlegendentry{$E^{N}_{\zeta(u)}$}
        \addplot [mark=o, mark options={solid,color=magenta}, magenta,dashed] table[x=C,y=L2z_L] {ep/eph.0625dt0.1.dat};
        \addlegendentry{$E^{L}_{\zeta(u)}$}
        \addplot [mark=diamond, mark options={solid,color=teal}, teal,dashed] table[x=C,y=L2z_R] {ep/eph.0625dt0.1.dat};
        \addlegendentry{$E^{R}_{\zeta(u)}$}
          \addplot [mark=asterisk, mark options={solid,color=cyan}, cyan] table[x=C,y=H1z_N] {ep/eph.0625dt0.1.dat};
        \addlegendentry{$E^{N}_{\nabla\zeta(u)}$}
        \addplot [mark=o, mark options={solid,color=magenta}, magenta] table[x=C,y=H1z_L] {ep/eph.0625dt0.1.dat};
       \addlegendentry{$E^{L}_{\nabla\zeta(u)}$}
        \addplot [mark=diamond, mark options={solid,color=teal}, teal] table[x=C,y=H1z_R] {ep/eph.0625dt0.1.dat};
       \addlegendentry{$E^{R}_{\nabla\zeta(u)}$}
        \end{loglogaxis}  
        \hskip 25pt
       \begin{loglogaxis} [ymin=0.001, ymax=0.3,name=plot2,at={(plot1.south east)},xlabel={$\delta t=0.01$}
      ]
        \addplot [mark=asterisk, mark options={solid,color=cyan}, cyan,dashed] table[x=C,y=L2z_N] {ep/eph.0625dt0.01.dat};
        \addplot [mark=o, mark options={solid,color=magenta}, magenta,dashed] table[x=C,y=L2z_L]  {ep/eph.0625dt0.01.dat};
        \addplot [mark=diamond, mark options={solid,color=teal}, teal,dashed] table[x=C,y=L2z_R] {ep/eph.0625dt0.01.dat};
          \addplot [mark=asterisk, mark options={solid,color=cyan}, cyan] table[x=C,y=H1z_N] {ep/eph.0625dt0.01.dat};
        \addplot [mark=o, mark options={solid,color=magenta}, magenta] table[x=C,y=H1z_L] {ep/eph.0625dt0.01.dat};
        \addplot [mark=diamond, mark options={solid,color=teal}, teal] table[x=C,y=H1z_R] {ep/eph.0625dt0.01.dat};
      \end{loglogaxis}
    \end{tikzpicture}
    \begin{tikzpicture}[scale=0.7]
      \begin{loglogaxis} [
      ymin=0.001, ymax=0.3,
      name=plot1, ytick pos=left,
      legend columns=6, 
      ylabel={$h=0.0313$},xlabel={$\delta t=0.1$}
      ]
        \addplot [mark=asterisk, mark options={solid,color=cyan}, cyan,dashed] table[x=C,y=L2z_N] {ep/eph.0313dt0.1.dat};
        \addplot [mark=o, mark options={solid,color=magenta}, magenta,dashed] table[x=C,y=L2z_L] {ep/eph.0313dt0.1.dat};
        \addplot [mark=diamond, mark options={solid,color=teal}, teal,dashed] table[x=C,y=L2z_R] {ep/eph.0313dt0.1.dat};
          \addplot [mark=asterisk, mark options={solid,color=cyan}, cyan] table[x=C,y=H1z_N] {ep/eph.0313dt0.1.dat};
        \addplot [mark=o, mark options={solid,color=magenta}, magenta] table[x=C,y=H1z_L] {ep/eph.0313dt0.1.dat};
        \addplot [mark=diamond, mark options={solid,color=teal}, teal] table[x=C,y=H1z_R] {ep/eph.0313dt0.1.dat};
        \end{loglogaxis}  
        \hskip 25pt
       \begin{loglogaxis} [ymin=0.001, ymax=0.3,name=plot2,at={(plot1.south east)},xlabel={$\delta t=0.01$}
      ]
        \addplot [mark=asterisk, mark options={solid,color=cyan}, cyan,dashed] table[x=C,y=L2z_N] {ep/eph.0313dt0.01.dat};
        \addplot [mark=o, mark options={solid,color=magenta}, magenta,dashed] table[x=C,y=L2z_L]  {ep/eph.0313dt0.01.dat};
        \addplot [mark=diamond, mark options={solid,color=teal}, teal,dashed] table[x=C,y=L2z_R] {ep/eph.0313dt0.01.dat};
          \addplot [mark=asterisk, mark options={solid,color=cyan}, cyan] table[x=C,y=H1z_N] {ep/eph.0313dt0.01.dat};
        \addplot [mark=o, mark options={solid,color=magenta}, magenta] table[x=C,y=H1z_L] {ep/eph.0313dt0.01.dat};
        \addplot [mark=diamond, mark options={solid,color=teal}, teal] table[x=C,y=H1z_R] {ep/eph.0313dt0.01.dat};
      \end{loglogaxis}
    \end{tikzpicture}
    \begin{tikzpicture}[scale=0.7]
      \begin{loglogaxis} [
      ymin=0.001, ymax=0.3,
      name=plot1, ytick pos=left,
      legend columns=6, 
      ylabel={$h=0.0156$},xlabel={$\delta t=0.1$}
      ]
        \addplot [mark=asterisk, mark options={solid,color=cyan}, cyan,dashed] table[x=C,y=L2z_N] {ep/eph.0156dt0.1.dat};
        \addplot [mark=o, mark options={solid,color=magenta}, magenta,dashed] table[x=C,y=L2z_L] {ep/eph.0156dt0.1.dat};
        \addplot [mark=diamond, mark options={solid,color=teal}, teal,dashed] table[x=C,y=L2z_R] {ep/eph.0156dt0.1.dat};
          \addplot [mark=asterisk, mark options={solid,color=cyan}, cyan] table[x=C,y=H1z_N] {ep/eph.0156dt0.1.dat};
        \addplot [mark=o, mark options={solid,color=magenta}, magenta] table[x=C,y=H1z_L] {ep/eph.0156dt0.1.dat};
        \addplot [mark=diamond, mark options={solid,color=teal}, teal] table[x=C,y=H1z_R] {ep/eph.0156dt0.1.dat};
        \end{loglogaxis}  
        \hskip 25pt
       \begin{loglogaxis} [ymin=0.001, ymax=0.3,name=plot2,at={(plot1.south east)},xlabel={$\delta t=0.01$}
      ]
        \addplot [mark=asterisk, mark options={solid,color=cyan}, cyan,dashed] table[x=C,y=L2z_N] {ep/eph.0156dt0.01.dat};
        \addplot [mark=o, mark options={solid,color=magenta}, magenta,dashed] table[x=C,y=L2z_L]  {ep/eph.0156dt0.01.dat};
        \addplot [mark=diamond, mark options={solid,color=teal}, teal,dashed] table[x=C,y=L2z_R] {ep/eph.0156dt0.01.dat};
          \addplot [mark=asterisk, mark options={solid,color=cyan}, cyan] table[x=C,y=H1z_N] {ep/eph.0156dt0.01.dat};
        \addplot [mark=o, mark options={solid,color=magenta}, magenta] table[x=C,y=H1z_L] {ep/eph.0156dt0.01.dat};
        \addplot [mark=diamond, mark options={solid,color=teal}, teal] table[x=C,y=H1z_R] {ep/eph.0156dt0.01.dat};
      \end{loglogaxis}
    \end{tikzpicture}
    \begin{tikzpicture}[scale=0.7]
      \begin{loglogaxis} [
      ymin=0.001, ymax=0.3,
      name=plot1, ytick pos=left,
      legend columns=6,
      ylabel={$h=0.0078$},xlabel={$\delta t=0.1$}
      ]
        \addplot [mark=asterisk, mark options={solid,color=cyan}, cyan,dashed] table[x=C,y=L2z_N] {ep/eph.0078dt0.1.dat};
        \addplot [mark=o, mark options={solid,color=magenta}, magenta,dashed] table[x=C,y=L2z_L] {ep/eph.0078dt0.1.dat};
        \addplot [mark=diamond, mark options={solid,color=teal}, teal,dashed] table[x=C,y=L2z_R] {ep/eph.0078dt0.1.dat};
          \addplot [mark=asterisk, mark options={solid,color=cyan}, cyan] table[x=C,y=H1z_N] {ep/eph.0078dt0.1.dat};
        \addplot [mark=o, mark options={solid,color=magenta}, magenta] table[x=C,y=H1z_L] {ep/eph.0078dt0.1.dat};
        \addplot [mark=diamond, mark options={solid,color=teal}, teal] table[x=C,y=H1z_R] {ep/eph.0078dt0.1.dat};
        \end{loglogaxis}  
        \hskip 25pt
       \begin{loglogaxis} [ymin=0.001, ymax=0.3,name=plot2,at={(plot1.south east)},xlabel={$\delta t=0.01$}
      ]
        \addplot [mark=asterisk, mark options={solid,color=cyan}, cyan,dashed] table[x=C,y=L2z_N] {ep/eph.0078dt0.01.dat};
        \addplot [mark=o, mark options={solid,color=magenta}, magenta,dashed] table[x=C,y=L2z_L]  {ep/eph.0078dt0.01.dat};
        \addplot [mark=diamond, mark options={solid,color=teal}, teal,dashed] table[x=C,y=L2z_R] {ep/eph.0078dt0.01.dat};
          \addplot [mark=asterisk, mark options={solid,color=cyan}, cyan] table[x=C,y=H1z_N] {ep/eph.0078dt0.01.dat};
        \addplot [mark=o, mark options={solid,color=magenta}, magenta] table[x=C,y=H1z_L] {ep/eph.0078dt0.01.dat};
        \addplot [mark=diamond, mark options={solid,color=teal}, teal] table[x=C,y=H1z_R] {ep/eph.0078dt0.01.dat};
      \end{loglogaxis}
    \end{tikzpicture}
  \caption{ Relative errors (on $y$-axis) for $\zeta(u)$ with respect to exact solution in $L^2$ (above) and $H_1$ (below) norms with Newton, Linearised and regularised solvers vs $C_{\varepsilon}$ ( on $x$-axis)}
\label{fig:ep_mesh1_5}
\end{figure}%
On the other hand, the case of $C_{\varepsilon}=1$ shows a negligible impact on the errors, suggesting that this lower value of the $C_{\varepsilon}$ does not significantly influence the overall error in the solution. These plots do not include the case $dt=0.001$ because their graphs were overlapping with no observable differences. Since, $C_{\varepsilon}=1$ does not lead to a significant change in the errors and larger values of $C_{\varepsilon}$ increases the convergence rate and consequently the computational time, we will compare these methods based on computational time with $C_{\text{tol}}= \{1, 10, 100\}$ and $C_{\varepsilon}= \{1\}$ (which means the $\varepsilon$ is of order $\delta t$ in most of the case) in the next experiments.%
\begin{figure}\centering
\centerline{ \ref{EZ}}
\vspace{0.50cm}
\begin{minipage}{0.45\textwidth}
\begin{tikzpicture}[scale=0.75]
\begin{loglogaxis}[name=egraphEZ4,legend columns=3,legend to name=EZ, 
        xlabel=$h$,
         ylabel=$E_\zeta$,x post scale =1
         ]         
         \addplot[cyan,mark=*, mark options={solid},] 
        table[x=h, y=NEZ] {chapters/dat_N/egraphC100ndt10.dat};
    \addlegendentry{$C_{tol}^N=100$}
        \addplot[magenta,thick,mark=*, mark options={solid},dashed] 
        table[x=h, y=LEZ] {chapters/dat_N/egraphC100ndt10.dat};
           \addlegendentry{$C_{tol}^L=100$}
             \addplot[teal,thick,mark=*, mark options={solid},dotted] 
         table[x=h, y=REZ] {chapters/dat_N/egraphC100ndt10.dat};
         \addlegendentry{$C_{tol}^R=100$}
         \addplot[cyan,mark=star, mark options={solid}] 
        table[x=h, y=NEZ] {chapters/dat_N/egraphC10ndt10.dat};
         \addlegendentry{$C_{tol}^N=10$}
        \addplot[magenta,thick,mark=star, mark options={solid},dashed] 
        table[x=h, y=LEZ] {chapters/dat_N/egraphC10ndt10.dat};
         \addlegendentry{$C_{tol}^L=10$}
             \addplot[teal,thick,mark=star, mark options={solid},dotted] 
         table[x=h, y=REZ] {chapters/dat_N/egraphC10ndt10.dat};
           \addlegendentry{$C_{tol}^R=10$}
          \addplot[cyan,mark=o, mark options={solid},] 
        table[x=h, y=NEZ] {chapters/dat_N/egraphC1ndt10.dat};
         \addlegendentry{$C_{tol}^N=1$}
        \addplot[magenta,thick,mark=o, mark options={solid},dashed] 
        table[x=h, y=LEZ] {chapters/dat_N/egraphC1ndt10.dat};
         \addlegendentry{$C_{tol}^L=1$}
             \addplot[teal,thick,mark=o, mark options={solid},dotted] 
         table[x=h, y=REZ] {chapters/dat_N/egraphC1ndt10.dat};
    \addlegendentry{$C_{tol}^R=1$}
               \end{loglogaxis}
\end{tikzpicture}
\subcaption{$\delta t=10^{-1}$}
\end{minipage}
\hskip 30pt
\begin{minipage}{0.45\textwidth}
\begin{tikzpicture}[scale=0.75]
        \begin{loglogaxis}[name=egraph3,
        xlabel=$h$,
         ylabel=$E_\zeta$,x post scale =1]
        \addplot[cyan,mark=*, mark options={solid},] 
        table[x=h, y=NEZ] {chapters/dat_N/egraphC100ndt100.dat};
        
        \addplot[magenta,thick,mark=*, mark options={solid},dashed] 
        table[x=h, y=LEZ] {chapters/dat_N/egraphC100ndt100.dat};
        
         \addplot[teal,thick,mark=*, mark options={solid},dotted] 
         table[x=h, y=REZ] {chapters/dat_N/egraphC100ndt100.dat};
       
        \addplot[cyan,mark=star, mark options={solid},] 
        table[x=h, y=NEZ] {chapters/dat_N/egraphC10ndt100.dat};
     
        \addplot[magenta,thick,mark=star, mark options={solid},dashed] 
        table[x=h, y=LEZ] {chapters/dat_N/egraphC10ndt100.dat};
       
         \addplot[teal,thick,mark=star, mark options={solid},dotted] 
         table[x=h, y=REZ] {chapters/dat_N/egraphC10ndt100.dat};
        
         \addplot[cyan,mark=o, mark options={solid},] 
        table[x=h, y=NEZ] {chapters/dat_N/egraphC1ndt100.dat};
   
        \addplot[magenta,thick,mark=o, mark options={solid},dashed] 
        table[x=h, y=LEZ] {chapters/dat_N/egraphC1ndt100.dat};
  
         \addplot[teal,thick,mark=o, mark options={solid},dotted] 
         table[x=h, y=REZ] {chapters/dat_N/egraphC1ndt100.dat};

        \end{loglogaxis}
\end{tikzpicture}
\subcaption{$\delta t=10^{-2}$}
\end{minipage}
\begin{minipage}{0.45\textwidth}
\begin{tikzpicture}[scale=0.75]
        \begin{loglogaxis}[name=egraph2,
        xlabel=$h$,
         ylabel=$E_\zeta$,x post scale =1
         ]         
         \addplot[cyan,mark=*, mark options={solid},] 
        table[x=h, y=NEZ] {chapters/dat_N/egraphC100ndt1000.dat};

        \addplot[magenta,thick,mark=*, mark options={solid},dashed] 
        table[x=h, y=LEZ] {chapters/dat_N/egraphC100ndt1000.dat};

             \addplot[teal,thick,mark=*, mark options={solid},dotted] 
         table[x=h, y=REZ] {chapters/dat_N/egraphC100ndt1000.dat};
    
         \addplot[cyan,mark=star, mark options={solid},] 
        table[x=h, y=NEZ] {chapters/dat_N/egraphC10ndt1000.dat};
       
        \addplot[magenta,thick,mark=star, mark options={solid},dashed] 
        table[x=h, y=LEZ] {chapters/dat_N/egraphC10ndt1000.dat};
       
             \addplot[teal,thick,mark=star, mark options={solid},dotted] 
         table[x=h, y=REZ] {chapters/dat_N/egraphC10ndt1000.dat};
          \addplot[cyan,mark=o, mark options={solid},] 
        table[x=h, y=NEZ] {chapters/dat_N/egraphC1ndt1000.dat};
       
        \addplot[magenta,thick,mark=o, mark options={solid},dashed] 
        table[x=h, y=LEZ] {chapters/dat_N/egraphC1ndt1000.dat};
       
             \addplot[teal,thick,mark=o, mark options={solid},dotted] 
         table[x=h, y=REZ] {chapters/dat_N/egraphC1ndt1000.dat};
   
               \end{loglogaxis}
\end{tikzpicture}
\subcaption{$\delta t=10^{-3}$}
\end{minipage}
\hskip 30pt
\begin{minipage}{0.45\textwidth}
        \begin{tikzpicture}[scale=0.75]
        \begin{loglogaxis}[name=egraph1,
        xlabel=$h$,
         ylabel=$E_\zeta$,x post scale =1
         ]         
         \addplot[cyan,mark=*, mark options={solid},] 
        table[x=h, y=NEZ] {chapters/dat_N/egraphC100ndt10000.dat};
       
        \addplot[magenta,thick,mark=*, mark options={solid},dashed] 
        table[x=h, y=LEZ] {chapters/dat_N/egraphC100ndt10000.dat};
        
             \addplot[teal,thick,mark=*, mark options={solid},dotted] 
         table[x=h, y=REZ] {chapters/dat_N/egraphC100ndt10000.dat};
        
         \addplot[cyan,mark=star, mark options={solid},] 
        table[x=h, y=NEZ] {chapters/dat_N/egraphC10ndt10000.dat};
       
        \addplot[magenta,thick,mark=star, mark options={solid},dashed] 
        table[x=h, y=LEZ] {chapters/dat_N/egraphC10ndt10000.dat};
       
             \addplot[teal,thick,mark=star, mark options={solid},dotted] 
         table[x=h, y=REZ] {chapters/dat_N/egraphC10ndt10000.dat};
          \addplot[cyan,mark=o, mark options={solid},] 
        table[x=h, y=NEZ] {chapters/dat_N/egraphC1ndt10000.dat};
       
        \addplot[magenta,thick,mark=o, mark options={solid},dashed] 
        table[x=h, y=LEZ] {chapters/dat_N/egraphC1ndt10000.dat};
       
             \addplot[teal,thick,mark=o, mark options={solid},mark options={solid},dotted] 
         table[x=h, y=REZ] {chapters/dat_N/egraphC1ndt10000.dat};
               \end{loglogaxis}
\end{tikzpicture}
\subcaption{$\delta t=10^{-4}$}
\end{minipage}
\caption{Comparison of $E_{\zeta}$ for different choices of $C_{tol}$ and $C_{\varepsilon}=1$}
\label{fig:Egraph_Ezeta}
 \end{figure}
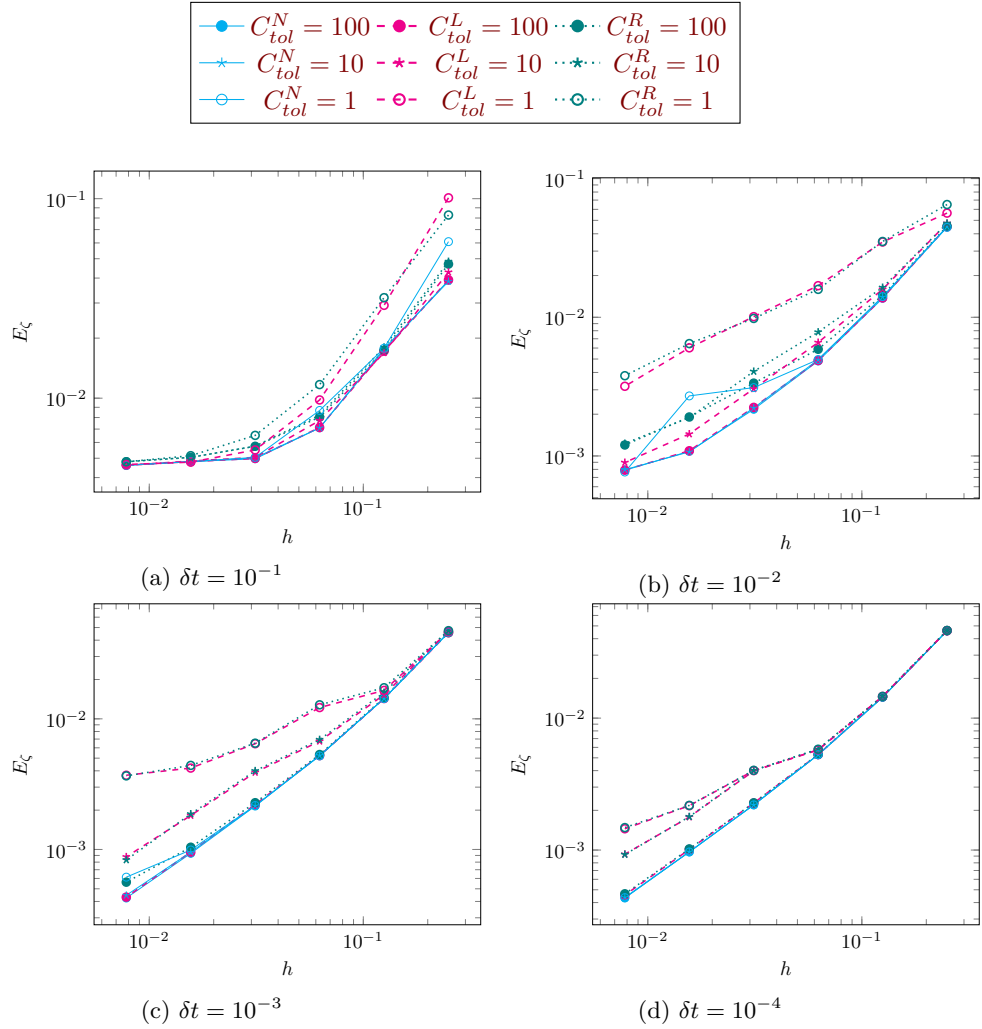
\begin{figure}\centering
\centerline{ \ref{EHZ}}
\vspace{0.50cm}
\begin{minipage}{0.45\textwidth}
\begin{tikzpicture}[scale=0.75]
        \begin{loglogaxis}[name=egraphEHZ4,legend columns=3,legend to name=EHZ, 
        xlabel=$h$,
         ylabel=$E_{\nabla\zeta}$,x post scale =1
         ]         
         \addplot[cyan,mark=*, mark options={solid},] 
        table[x=h, y=NEH] {chapters/dat_N/egraphC100ndt10.dat};
    \addlegendentry{$C_{tol}^N=100$}
        \addplot[magenta,thick,mark=*, mark options={solid},dashed] 
        table[x=h, y=LEH] {chapters/dat_N/egraphC100ndt10.dat};
           \addlegendentry{$C_{tol}^L=100$}
             \addplot[teal,thick,mark=*, mark options={solid},dotted] 
         table[x=h, y=REH] {chapters/dat_N/egraphC100ndt10.dat};
         \addlegendentry{$C_{tol}^R=100$}
         \addplot[cyan,mark=star, mark options={solid}] 
        table[x=h, y=NEH] {chapters/dat_N/egraphC10ndt10.dat};
         \addlegendentry{$C_{tol}^N=10$}
        \addplot[magenta,thick,mark=star, mark options={solid},dashed] 
        table[x=h, y=LEH] {chapters/dat_N/egraphC10ndt10.dat};
         \addlegendentry{$C_{tol}^L=10$}
             \addplot[teal,thick,mark=star, mark options={solid},dotted] 
         table[x=h, y=REH] {chapters/dat_N/egraphC10ndt10.dat};
           \addlegendentry{$C_{tol}^R=10$}
          \addplot[cyan,mark=o, mark options={solid},] 
        table[x=h, y=NEH] {chapters/dat_N/egraphC1ndt10.dat};
         \addlegendentry{$C_{tol}^N=1$}
        \addplot[magenta,thick,mark=o, mark options={solid},dashed] 
        table[x=h, y=LEH] {chapters/dat_N/egraphC1ndt10.dat};
         \addlegendentry{$C_{tol}^L=1$}
             \addplot[teal,thick,mark=o, mark options={solid},dotted] 
         table[x=h, y=REH] {chapters/dat_N/egraphC1ndt10.dat};
    \addlegendentry{$C_{tol}^R=1$}
               \end{loglogaxis}
\end{tikzpicture}
\subcaption{$\delta t=10^{-1}$}
\end{minipage}
\hskip 30pt
\begin{minipage}{0.45\textwidth}
\begin{tikzpicture}[scale=0.75]
        \begin{loglogaxis}[name=egraphEHZ3,
        xlabel=$h$,
         ylabel=$E_{\nabla\zeta}$,x post scale =1]
        \addplot[cyan,mark=*, mark options={solid},] 
        table[x=h, y=NEH] {chapters/dat_N/egraphC100ndt100.dat};
        
        \addplot[magenta,thick,mark=*, mark options={solid},dashed] 
        table[x=h, y=LEH] {chapters/dat_N/egraphC100ndt100.dat};
        
         \addplot[teal,thick,mark=*, mark options={solid},dotted] 
         table[x=h, y=REH] {chapters/dat_N/egraphC100ndt100.dat};
       
        \addplot[cyan,mark=star, mark options={solid},] 
        table[x=h, y=NEH] {chapters/dat_N/egraphC10ndt100.dat};
     
        \addplot[magenta,thick,mark=star, mark options={solid},dashed] 
        table[x=h, y=LEH] {chapters/dat_N/egraphC10ndt100.dat};
       
         \addplot[teal,thick,mark=star, mark options={solid},dotted] 
         table[x=h, y=REH] {chapters/dat_N/egraphC10ndt100.dat};
        
         \addplot[cyan,mark=o, mark options={solid},] 
        table[x=h, y=NEH] {chapters/dat_N/egraphC1ndt100.dat};
   
        \addplot[magenta,thick,mark=o, mark options={solid},dashed] 
        table[x=h, y=LEH] {chapters/dat_N/egraphC1ndt100.dat};
  
         \addplot[teal,thick,mark=o, mark options={solid},dotted] 
         table[x=h, y=REH] {chapters/dat_N/egraphC1ndt100.dat};

               \end{loglogaxis}
\end{tikzpicture}
\subcaption{$\delta t=10^{-2}$}
\end{minipage}
\begin{minipage}{0.45\textwidth}
\begin{tikzpicture}[scale=0.75]
        \begin{loglogaxis}[name=egraphEHZ2,
        xlabel=$h$,
         ylabel=$E_{\nabla\zeta}$,x post scale =1
         ]         
         \addplot[cyan,mark=*, mark options={solid},] 
        table[x=h, y=NEH] {chapters/dat_N/egraphC100ndt1000.dat};

        \addplot[magenta,thick,mark=*, mark options={solid},dashed] 
        table[x=h, y=LEH] {chapters/dat_N/egraphC100ndt1000.dat};

             \addplot[teal,thick,mark=*, mark options={solid},dotted] 
         table[x=h, y=REH] {chapters/dat_N/egraphC100ndt1000.dat};
    
         \addplot[cyan,mark=star, mark options={solid},] 
        table[x=h, y=NEH] {chapters/dat_N/egraphC10ndt1000.dat};
       
        \addplot[magenta,thick,mark=star, mark options={solid},dashed] 
        table[x=h, y=LEH] {chapters/dat_N/egraphC10ndt1000.dat};
       
             \addplot[teal,thick,mark=star, mark options={solid},dotted] 
         table[x=h, y=REH] {chapters/dat_N/egraphC10ndt1000.dat};
          \addplot[cyan,mark=o, mark options={solid},] 
        table[x=h, y=NEH] {chapters/dat_N/egraphC1ndt1000.dat};
       
        \addplot[magenta,thick,mark=o, mark options={solid},dashed] 
        table[x=h, y=LEH] {chapters/dat_N/egraphC1ndt1000.dat};
       
             \addplot[teal,thick,mark=o, mark options={solid},dotted] 
         table[x=h, y=REH] {chapters/dat_N/egraphC1ndt1000.dat};
   
               \end{loglogaxis}
\end{tikzpicture}
\subcaption{$\delta t=10^{-3}$}
\end{minipage}
\hskip 30pt
\begin{minipage}{0.45\textwidth}
\begin{tikzpicture}[scale=0.75]
        \begin{loglogaxis}[name=egraphEHZ1,
        xlabel=$h$,
         ylabel=$E_{\nabla\zeta}$,x post scale =1
         ]         
         \addplot[cyan,mark=*, mark options={solid},] 
        table[x=h, y=NEH] {chapters/dat_N/egraphC100ndt10000.dat};
       
        \addplot[magenta,thick,mark=*, mark options={solid},dashed] 
        table[x=h, y=LEH] {chapters/dat_N/egraphC100ndt10000.dat};
        
             \addplot[teal,thick,mark=*, mark options={solid},dotted] 
         table[x=h, y=REH] {chapters/dat_N/egraphC100ndt10000.dat};
        
         \addplot[cyan,mark=star, mark options={solid},] 
        table[x=h, y=NEH] {chapters/dat_N/egraphC10ndt10000.dat};
       
        \addplot[magenta,thick,mark=star, mark options={solid},dashed] 
        table[x=h, y=LEH] {chapters/dat_N/egraphC10ndt10000.dat};
       
             \addplot[teal,thick,mark=star, mark options={solid},dotted] 
         table[x=h, y=REH] {chapters/dat_N/egraphC10ndt10000.dat};
          \addplot[cyan,mark=o, mark options={solid},] 
        table[x=h, y=NEH] {chapters/dat_N/egraphC1ndt10000.dat};
       
        \addplot[magenta,thick,mark=o, mark options={solid},dashed] 
        table[x=h, y=LEH] {chapters/dat_N/egraphC1ndt10000.dat};
       
             \addplot[teal,thick,mark=o, mark options={solid},mark options={solid},dotted] 
         table[x=h, y=REH] {chapters/dat_N/egraphC1ndt10000.dat};
   
               \end{loglogaxis}
\end{tikzpicture}
\subcaption{$\delta t=10^{-4}$}
\end{minipage}
\caption{Comparison of $E_{\nabla\zeta}$ for different choices of $C_{tol}$ and $C_{\varepsilon}=1$ }
\label{fig:Egraph_Enablazeta}
 \end{figure}

The above-mentioned figures can be recast to examine the convergence of the error norms. Figures \ref{fig:Egraph_Ezeta} and \ref{fig:Egraph_Enablazeta} illustrate the convergence of $E_\zeta$ and $E_{\nabla\zeta}$ for fixed time steps as the mesh size $h$ is refined across various choices of $C_{\text{tol}}$ with $C_{\varepsilon} = 1$. In the first figure, the results show consistent convergence of $E_\zeta$ for all methods, with similar values across the range of tolerances. This suggests choosing $C_{\text{tol}}=1$ since the error is satisfactorily with slow convergence rate as expected. The second figure, which focuses on the error $E_{\nabla\zeta}$, also shows similar behaviour across all solvers when $C_{\text{tol}} = 100$. However, the linearised (L) and regularised (R) methods struggle to maintain convergence for smaller values of $C_{\text{tol}}$, particularly for $\delta t = 10^{-3}$, where they fail to show convergence on finer meshes. Moreover, the Newton method also begins to diverge, which is visible in the case of $\delta t = \{10^{-2},10^{-3}\}$, indicating that even the Newton method has limitations under certain conditions. This is not a surprise in approximating the gradient where a higher tolerance is required. In fact, the convergence of the L scheme is proved in Lemma \ref{lem:conv_unreg} only in the weak norm $\| \cdot \|_{*,\mc{D}}$, thus not for the gradient. For the regularised scheme in Section \ref{sec:reg}, the contraction is proved in the $\| \cdot \|_{X}$ norm, where the gradient component is weighted by $\delta t^{1/2}$, which explains the limited effect of the iterations on the gradient component of the error. These observations suggest that for ensuring convergence in gradient errors, a stricter tolerance ($C_{\text{tol}} \geq 100$) may be necessary, particularly when working with finer meshes and small time steps. The divergence observed in $E_{\nabla\zeta}$ for smaller values of $C_{\text{tol}}$ may be linked to the approximation of $\nabla\zeta$ with $L$ to obtain a fixed left-hand side for the linearised system.

In light of the previous experiments, we now examine the computational efficiency, beginning with the deterministic case and then moving on to the stochastic case. Figures \ref{fig:dt100e1}, \ref{fig:dt10} and \ref{fig:dt1} shows a comparison of cumulative CPU time for N, L and R methods for various values of $C_{\text{tol}}$ in decreasing order with fixed $C_{\varepsilon}=1$. The time is calculated by taking the minimum of ten repetitions of each simulation. On the x-axis, "\textsc{Sn}" represents a simulation number which start with the coarsest mesh (M1) simulating for each $N_{\delta t}=\{1,10,10^2,10^3,10^4\}$ and goes up to the finest mesh (M6). In all of these graphs, L and R methods are outperforming N  (as can be seen in the accompanied subplots) up to mesh M5. However, Newton has a significant edge in M6 specifically for $C_{\text{tol}}=100$ whereas all the methods taking almost same CPU time when $C_{\text{tol}}=10$. The choice $C_{\text{tol}}=1$ goes in the favour of L and R methods, which suits best in computing $E_\zeta$ as discussed above.
\begin{figure}\centering
\centerline{\ref{Dt100e102}}
\vspace{0.2cm}
\begin{tikzpicture}[scale=0.85]
        \begin{axis}[name=firstfigT100E10,legend columns=3,legend to name=Dt100e102, 
        xlabel=\textsc{Sn},
         ylabel=\textsc{CPU time},x post scale =2.5
         ]         
        \addplot[cyan,no markers] 
        table[x=sr, y=NCC] {chapters/dat_N/D_compt100.dat};
        \addlegendentry{N}
        \addplot[magenta,thick,no markers,dashed] 
        table[x=sr, y=LCC] {chapters/dat_N/D_compt100.dat};
         \addlegendentry{L}
         \addplot[teal,thick,no markers,dotted] 
         table[x=sr, y=RCC] {chapters/dat_N/D_compt100.dat};
        \addlegendentry{R}
          \end{axis}
\node[pin=92:{%
    \begin{tikzpicture}[scale=1.2]
    \begin{axis}[tiny,
    domain=1:7,
      xmin=2,xmax=6,
      ymin=0,ymax=10,
      line join=round,
      enlargelimits,width = 4cm,
      tick align=outside
      ,skip coords between index={8}{24}%
    ]
      \addplot[cyan,no markers] 
        table[x=sr, y=NCC] {chapters/dat_N/D_compt100.dat};
        \addplot[magenta,thick,no markers,dashed] 
        table[x=sr, y=LCC] {chapters/dat_N/D_compt100.dat};
         \addplot[teal,thick,no markers,dotted] 
         table[x=sr, y=RCC] {chapters/dat_N/D_compt100.dat};
    \end{axis}
    \end{tikzpicture}%
}] at (3,0.5) {};
    \node[pin=90:{%
    \begin{tikzpicture}[scale=1.2]
    \begin{axis}[tiny,
    domain=8:15,
      xmin=9,xmax=14,
      ymin=18,ymax=120,
      line join=round,
      enlargelimits,width = 4cm,
      tick align=outside
      ,skip coords between index={16}{24}%
    ]
      \addplot[cyan,no markers] 
        table[x=sr, y=NCC] {chapters/dat_N/D_compt100.dat};
        \addplot[magenta,thick,no markers,dashed] 
        table[x=sr, y=LCC] {chapters/dat_N/D_compt100.dat};
         \addplot[teal,thick,no markers,dotted] 
         table[x=sr, y=RCC] {chapters/dat_N/D_compt100.dat};
    \end{axis}
    \end{tikzpicture}%
}] at (9,0.5) {};
    \end{tikzpicture}
    \caption{Deterministic case: cumulative frequency graph of CPU times vs.\ time step for $C_{tol}=100$ and $C_{\varepsilon}=1$}
    \label{fig:dt100e1}
 \end{figure}%
\begin{figure}\centering
\centerline{ \ref{t10e11}}
\vspace{0.50cm}
\begin{tikzpicture}[scale=0.85]
        \begin{axis}[name=firstfigt10e1,legend columns=3,legend to name=t10e11, tick align=outside,tick pos=lower,
        xlabel=\textsc{Sn},
         ylabel=\textsc{CPU time},x post scale =2.5
         ]
         
        \addplot[cyan,no markers] 
        table[x=sr, y=NCC] {chapters/dat_N/D_compt10.dat};
        \addlegendentry{N}
        \addplot[magenta,thick,no markers,dashed] 
        table[x=sr, y=LCC] {chapters/dat_N/D_compt10.dat};
         \addlegendentry{L}
         \addplot[teal,thick,no markers,dotted] 
         table[x=sr, y=RCC] {chapters/dat_N/D_compt10.dat};
        \addlegendentry{R}
          \end{axis}
    \node[pin=90:{%
    \begin{tikzpicture}[scale=1.2]
    \begin{axis}[tiny,
     domain=1:6,
      xmin=2,xmax=5,
      ymin=0,ymax=6,
      line join=round,
      enlargelimits,width = 4cm,
      tick align=outside
      ,skip coords between index={8}{24}%
    ]
      \addplot[cyan,no markers] 
        table[x=sr, y=NCC] {chapters/dat_N/D_compt10.dat};
        \addplot[magenta,thick,no markers,dashed] 
        table[x=sr, y=LCC] {chapters/dat_N/D_compt10.dat};
         \addplot[teal,thick,no markers,dotted] 
         table[x=sr, y=RCC] {chapters/dat_N/D_compt10.dat};
    \end{axis}
    \end{tikzpicture}%
}] at (3,0.2) {};
    \node[pin=15:{%
    \begin{tikzpicture}[scale=1.2]
    \begin{axis}[tiny,
    domain=6:12,
      xmin=7,xmax=11,
      ymin=10,ymax=50,
      line join=round,
      enlargelimits,width = 4cm,
      tick align=outside
      ,skip coords between index={13}{24}%
    ]
      \addplot[cyan,no markers] 
        table[x=sr, y=NCC] {chapters/dat_N/D_compt10.dat};
        \addplot[magenta,thick,no markers,dashed] 
        table[x=sr, y=LCC] {chapters/dat_N/D_compt10.dat};
         \addplot[teal,thick,no markers,dotted] 
         table[x=sr, y=RCC] {chapters/dat_N/D_compt10.dat};
    \end{axis}
    \end{tikzpicture}%
}] at (6,0.5) {};
    \end{tikzpicture}
    \caption{Deterministic case: cumulative frequency graph of CPU times vs.\ time step for $C_{tol}=10$ and $C_{\varepsilon}=1$}
    \label{fig:dt10}
 \end{figure}%
\begin{figure}\centering
\centerline{ \ref{Dt100e101}}
\vspace{0.50cm}
\begin{tikzpicture}[scale=0.85]
        \begin{axis}[name=firstfigT100E10,legend columns=3,legend to name=Dt100e101, 
        xlabel=\textsc{Sn},
         ylabel=\textsc{CPU time},x post scale =2.5
         ]         
        \addplot[cyan,no markers] 
        table[x=sr, y=NCC] {chapters/dat_N/D_compt1.dat};
        \addlegendentry{N}
        \addplot[magenta,thick,no markers,dashed] 
        table[x=sr, y=LCC] {chapters/dat_N/D_compt1.dat};
         \addlegendentry{L}
         \addplot[teal,thick,no markers,dotted] 
         table[x=sr, y=RCC] {chapters/dat_N/D_compt1.dat};
        \addlegendentry{R}
          \end{axis}
\node[pin=92:{%
    \begin{tikzpicture}[scale=1.2]
    \begin{axis}[tiny,
    domain=2:9,
      xmin=3,xmax=8,
      ymin=0,ymax=30,
      line join=round,
      enlargelimits,width = 4cm,
      tick align=outside
      ,skip coords between index={10}{24}%
    ]
      \addplot[cyan,no markers] 
        table[x=sr, y=NCC] {chapters/dat_N/D_compt1.dat};
        \addplot[magenta,thick,no markers,dashed] 
        table[x=sr, y=LCC] {chapters/dat_N/D_compt1.dat};
         \addplot[teal,thick,no markers,dotted] 
         table[x=sr, y=RCC] {chapters/dat_N/D_compt1.dat};
    \end{axis}
    \end{tikzpicture}%
}] at (3.5,0.5) {};
    \node[pin=90:{%
    \begin{tikzpicture}[scale=1.2]
    \begin{axis}[tiny,
    domain=8:15,
      xmin=9,xmax=14,
      ymin=18,ymax=120,
      line join=round,
      enlargelimits,width = 4cm,
      tick align=outside
      ,skip coords between index={16}{24}%
    ]
      \addplot[cyan,no markers] 
        table[x=sr, y=NCC] {chapters/dat_N/D_compt1.dat};
        \addplot[magenta,thick,no markers,dashed] 
        table[x=sr, y=LCC] {chapters/dat_N/D_compt1.dat};
         \addplot[teal,thick,no markers,dotted] 
         table[x=sr, y=RCC] {chapters/dat_N/D_compt1.dat};
    \end{axis}
    \end{tikzpicture}%
}] at (9,0.5) {};
    \end{tikzpicture}
    \caption{Deterministic case: cumulative frequency graph of CPU times vs.\ time step for $C_{tol}=1$ and $C_{\varepsilon}=1$}
    \label{fig:dt1}
 \end{figure}%

\subsection{The stochastic case}

We now address the stochastic case of the Stefan problem. For a detailed implementation, refer to \cite{DRONIOU2024114}. Numerically, the stochastic PDE is solved by simulating the equation over a finite set of Brownian motion samples. The expectation of the resulting solutions provides an approximate solution to the stochastic problem. The first experiment examines how CPU time changes for each method as the number of Brownian motions (nbm) increases. Table~\ref{tab:BM_scaling} presents the CPU time (in seconds) required for methods N, L, and R across different numbers of time steps ($\delta t$) and varying numbers of Brownian motions (BM) on a fixed mesh (\texttt{M4}). The results indicate that the CPU time for each method scales linearly with the number of Brownian motions. For example, when comparing the CPU time for BM = 1 with BM = 10, the time increases by a factor of approximately 10, and similarly, when BM = 100, the time increases by nearly a factor of 100. This shows that the time difference between any two methods (for any fix $h$ and $\delta t$) will also increase by this scale (roughly).%
\begin{table}
\begin{tabular}{|l|lll|lll|lll|}
\hline
    M4   & \multicolumn{3}{c|}{$N_{\delta t}$=10}                                      & \multicolumn{3}{c|}{$N_{\delta t}$=100}                                       & \multicolumn{3}{c|}{$N_{\delta t}$=1000}                                         \\ \hline
BM     & \multicolumn{1}{l|}{N}     & \multicolumn{1}{l|}{L}      & R     & \multicolumn{1}{l|}{N}      & \multicolumn{1}{l|}{L}      & R      & \multicolumn{1}{l|}{N}       & \multicolumn{1}{l|}{L}       & R       \\ \hline
1   & \multicolumn{1}{l|}{0.9}  & \multicolumn{1}{l|}{1.2}   & 0.5  & \multicolumn{1}{l|}{3.5}   & \multicolumn{1}{l|}{4.3}   & 3.6   & \multicolumn{1}{l|}{22.9}   & \multicolumn{1}{l|}{20.03}   & 19.0   \\ \hline
10  & \multicolumn{1}{l|}{7.0}  & \multicolumn{1}{l|}{10.1}  & 4.4  & \multicolumn{1}{l|}{29.7}  & \multicolumn{1}{l|}{38.6}  & 32.3  & \multicolumn{1}{l|}{205.9}  & \multicolumn{1}{l|}{185.7}  & 178.2  \\ \hline
100 & \multicolumn{1}{l|}{77.99} & \multicolumn{1}{l|}{113.01} & 47.31 & \multicolumn{1}{l|}{310.59} & \multicolumn{1}{l|}{408} & 335.1 & \multicolumn{1}{l|}{2036.8} & \multicolumn{1}{l|}{1783.4} & 1701.4 \\ \hline
\end{tabular}
\caption{Comparison of CPU time across different Brownian motions for methods N, L, and R.}
\label{tab:BM_scaling}
\end{table}%

Similarly to the deterministic case, we conducted simulations for the stochastic scenario using the same set of parameters. In this case, each simulation (for a fixed $\delta t$ and $h$) was run for 10 random Brownian motions, resulting in a non-zero source term, unlike in the deterministic case. Each set of simulations was repeated 10 times to obtain the minimum CPU time for each method. The cumulative frequency of CPU times for each method is depicted in Figures \ref{fig:S_CPU_t100}, \ref{fig:S_CPU_t10} and \ref{fig:S_CPU_t1}. %
\begin{figure}\centering
\centerline{ \ref{t100e10c}}
\vspace{0.25cm}
\begin{tikzpicture}[scale=0.85]
        \begin{axis}[name=firstfig,legend columns=3,legend to name=t100e10c, 
        xlabel=\textsc{Sn},
         ylabel=\textsc{CPU time},x post scale =2.5
         ]         
        \addplot[cyan,no markers] 
        table[x=sr, y=NCC] {chapters/dat_N/S_CPU_t100.dat};
        \addlegendentry{N}
        \addplot[magenta,thick,no markers,dashed] 
        table[x=sr, y=LCC] {chapters/dat_N/S_CPU_t100.dat};
         \addlegendentry{L}
         \addplot[teal,thick,no markers,dotted] 
         table[x=sr, y=RCC] {chapters/dat_N/S_CPU_t100.dat};
        \addlegendentry{R}
          \end{axis}
    \node[pin=106:{%
    \begin{tikzpicture}[scale=1.2]
    \begin{axis}[tiny,
    xtick=\empty,
    extra x ticks={4,5,6,7},
    extra x tick labels={4,5,6,7},
    domain=3:8,
      xmin=4,xmax=7,
      ymin=20,ymax=60,
      line join=round,
      enlargelimits,width = 4cm,
      tick align=outside
      ,x post scale =1
      ,skip coords between index={8}{24}%
    ]
      \addplot[cyan,no markers] 
        table[x=sr, y=NCC] {chapters/dat_N/S_CPU_t100.dat};
        \addplot[magenta,thick,no markers,dashed] 
        table[x=sr, y=LCC] {chapters/dat_N/S_CPU_t100.dat};
         \addplot[teal,thick,no markers,dotted] 
         table[x=sr, y=RCC] {chapters/dat_N/S_CPU_t100.dat};
    \end{axis}
    \end{tikzpicture}%
}] at (5.4,0.3) {};
 \node[pin={[pin distance=0.4cm]88:{%
    \begin{tikzpicture}[scale=1.2]
    \begin{axis}[tiny,%
     domain=10:15,
     xmin=11,xmax=14,%
      ymin=150,ymax=1000,%
      line join=round,%
      enlargelimits,width = 4cm,%
      tick align=outside%
      ,x post scale =1.6%
       ,skip coords between index={16}{24}%
    ]%
      \addplot[cyan,no markers] %
        table[x=sr, y=NCC] {chapters/dat_N/S_CPU_t1.dat};%
        \addplot[magenta,thick,no markers,dashed]%
        table[x=sr, y=LCC] {chapters/dat_N/S_CPU_t1.dat};%
         \addplot[teal,thick,no markers,dotted]%
         table[x=sr, y=RCC] {chapters/dat_N/S_CPU_t1.dat};%
    \end{axis}
    \end{tikzpicture}%
}}] at (9.6,0.3) {};
    \end{tikzpicture}
    \caption{Stochastic case with 10 Brownian motions: cumulative frequency graph of CPU times vs.\ time step for $C_{tol}=100$ and $C_{\varepsilon}=1$}
    \label{fig:S_CPU_t100}
 \end{figure}%
\begin{figure}\centering
\centerline{ \ref{S_CPU_t10t100e10b}}
\vspace{0.25cm}
\begin{tikzpicture}[scale=0.85]
\begin{axis}[name=firstfigS_CPU_t10,legend columns=6,legend to name=S_CPU_t10t100e10b, tick align=outside,tick pos=lower,
        xlabel=\textsc{Sn},
         ylabel=\textsc{CPU time},x post scale =2.5
         ]
        \addplot[cyan,no markers] 
        table[x=sr, y=NCC] {chapters/dat_N/S_CPU_t10.dat};
        \addlegendentry{N}
        \addplot[magenta,thick,no markers,dashed] 
        table[x=sr, y=LCC] {chapters/dat_N/S_CPU_t10.dat};
         \addlegendentry{L}
         \addplot[teal,thick,no markers,dotted] 
         table[x=sr, y=RCC] {chapters/dat_N/S_CPU_t10.dat};
        \addlegendentry{R}
          \end{axis}
    \node[pin=106:{%
    \begin{tikzpicture}[scale=1.2]
    \begin{axis}[tiny,
    domain=3:8,
      xmin=4,xmax=7,
      ymin=15,ymax=60,
      line join=round,
      enlargelimits,width = 4cm,
      tick align=outside
      ,x post scale =1
      ,skip coords between index={8}{24}%
    ]
      \addplot[cyan,no markers] 
        table[x=sr, y=NCC] {chapters/dat_N/S_CPU_t1.dat};
        \addplot[magenta,thick,no markers,dashed] 
        table[x=sr, y=LCC] {chapters/dat_N/S_CPU_t1.dat};
         \addplot[teal,thick,no markers,dotted] 
         table[x=sr, y=RCC] {chapters/dat_N/S_CPU_t1.dat};
    \end{axis}
    \end{tikzpicture}%
}] at (5.4,0.3) {};
  \node[pin={[pin distance=0.4cm]88:{%
    \begin{tikzpicture}[scale=1.2]
    \begin{axis}[tiny,%
   domain=10:15,
     xmin=11,xmax=14,%
      ymin=150,ymax=1000,%
      line join=round,%
      enlargelimits,width = 4cm,%
      tick align=outside%
      ,x post scale =1.6%
      ,skip coords between index={16}{24}%
    ]%
      \addplot[cyan,no markers] %
        table[x=sr, y=NCC] {chapters/dat_N/S_CPU_t1.dat};%
        \addplot[magenta,thick,no markers,dashed]%
        table[x=sr, y=LCC] {chapters/dat_N/S_CPU_t1.dat};%
         \addplot[teal,thick,no markers,dotted]%
         table[x=sr, y=RCC] {chapters/dat_N/S_CPU_t1.dat};%
    \end{axis}
    \end{tikzpicture}%
}}] at (9.6,0.3) {};
      
        
       
    \end{tikzpicture}
    \caption{Stochastic case with 10 Brownian motions: Cumulative frequency graph of CPU times vs.\ time step for $C_{tol}=10$ and $C_{\varepsilon}=1$}
    \label{fig:S_CPU_t10}
 \end{figure}%
\begin{figure}\centering
  \centerline{ \ref{t100e10a}}
\vspace{0.25cm}
\begin{tikzpicture}[scale=0.85]
        \begin{axis}[name=firstfig,legend columns=3,legend to name=t100e10a, 
        xlabel=\textsc{Sr},
         ylabel=\textsc{CPU time},x post scale =2.5
         ]         
        \addplot[cyan,no markers] 
        table[x=sr, y=NCC] {chapters/dat_N/S_CPU_t1.dat};
        \addlegendentry{N}
        \addplot[magenta,thick,no markers,dashed] 
        table[x=sr, y=LCC] {chapters/dat_N/S_CPU_t1.dat};
         \addlegendentry{L}
         \addplot[teal,thick,no markers,dotted] 
         table[x=sr, y=RCC] {chapters/dat_N/S_CPU_t1.dat};
        \addlegendentry{R}
          \end{axis}
    \node[pin=104:{%
    \begin{tikzpicture}[scale=1.2]
    \begin{axis}[tiny,
    xtick=\empty,
    extra x ticks={4,5,6,7},
    extra x tick labels={4,5,6,7},
    domain=3:8,
      xmin=4,xmax=7,
      ymin=20,ymax=60,
      line join=round,
      enlargelimits,width = 4cm,
      tick align=outside
      ,x post scale =1
      ,skip coords between index={8}{24}%
    ]
      \addplot[cyan,no markers] 
        table[x=sr, y=NCC] {chapters/dat_N/S_CPU_t1.dat};
        \addplot[magenta,thick,no markers,dashed] 
        table[x=sr, y=LCC] {chapters/dat_N/S_CPU_t1.dat};
         \addplot[teal,thick,no markers,dotted] 
         table[x=sr, y=RCC] {chapters/dat_N/S_CPU_t1.dat};
    \end{axis}
    \end{tikzpicture}%
}] at (5.25,0.3) {};
  \node[pin={[pin distance=0.4cm]88:{%
    \begin{tikzpicture}[scale=1.2]
    \begin{axis}[tiny,%
   domain=10:15,
     xmin=11,xmax=14,%
      ymin=200,ymax=1000,%
      line join=round,%
      enlargelimits,width = 4cm,%
      tick align=outside%
      ,x post scale =1.6%
      ,skip coords between index={16}{24}%
    ]%
      \addplot[cyan,no markers] %
        table[x=sr, y=NCC] {chapters/dat_N/S_CPU_t1.dat};%
        \addplot[magenta,thick,no markers,dashed]%
        table[x=sr, y=LCC] {chapters/dat_N/S_CPU_t1.dat};%
         \addplot[teal,thick,no markers,dotted]%
         table[x=sr, y=RCC] {chapters/dat_N/S_CPU_t1.dat};%
    \end{axis}
    \end{tikzpicture}%
}}] at (9.55,0.3) {};
    \end{tikzpicture}
    \caption{Stochastic case with 10 Brownian motions: cumulative frequency graph of CPU times vs.\ time step for $C_{tol}=1$ and $C_{\varepsilon}=1$}
    \label{fig:S_CPU_t1}
 \end{figure}%
The overall behaviour is consistent with what was observed in the deterministic case; however, the CPU time difference is notably increased due to the number of Brownian motions. In Figure \ref{fig:S_CPU_t100}, L and R method show a substantial advantage over N except for the mesh M6 (which starts from \textsc{Sn}$=20$) where N outperforms them. It is important to note that in this extreme scenario, R performs better than L, highlighting the efficacy of the smaller $C_{\varepsilon}$ value chosen. As $C_{\text{tol}}$ decreases to 10 (which corresponds to an increase in tolerance), the L and R methods begin to converge in fewer iterations, resulting in improved efficiency. Since N is already saturated, its computational time remains relatively unaffected. This trend is even more pronounced in Figure \ref{fig:S_CPU_t1}, where $C_{\text{tol}}=1$. Here, L and R methods demonstrate significant improvements, with a reduction in CPU time (approximately 5.5 hours) compared to the N method.

Lastly, to test the performance in the absence of exact solution, we replicate “Test-1” from \cite[Section 6.1]{DRONIOU2024114} using N, L and R methods with $C_{\text{tol}}=\{1, 10, 100\}$ and $C_{\varepsilon}=1$. The methods are compared in Figure \ref{fig:Egraph_Ezeta} %
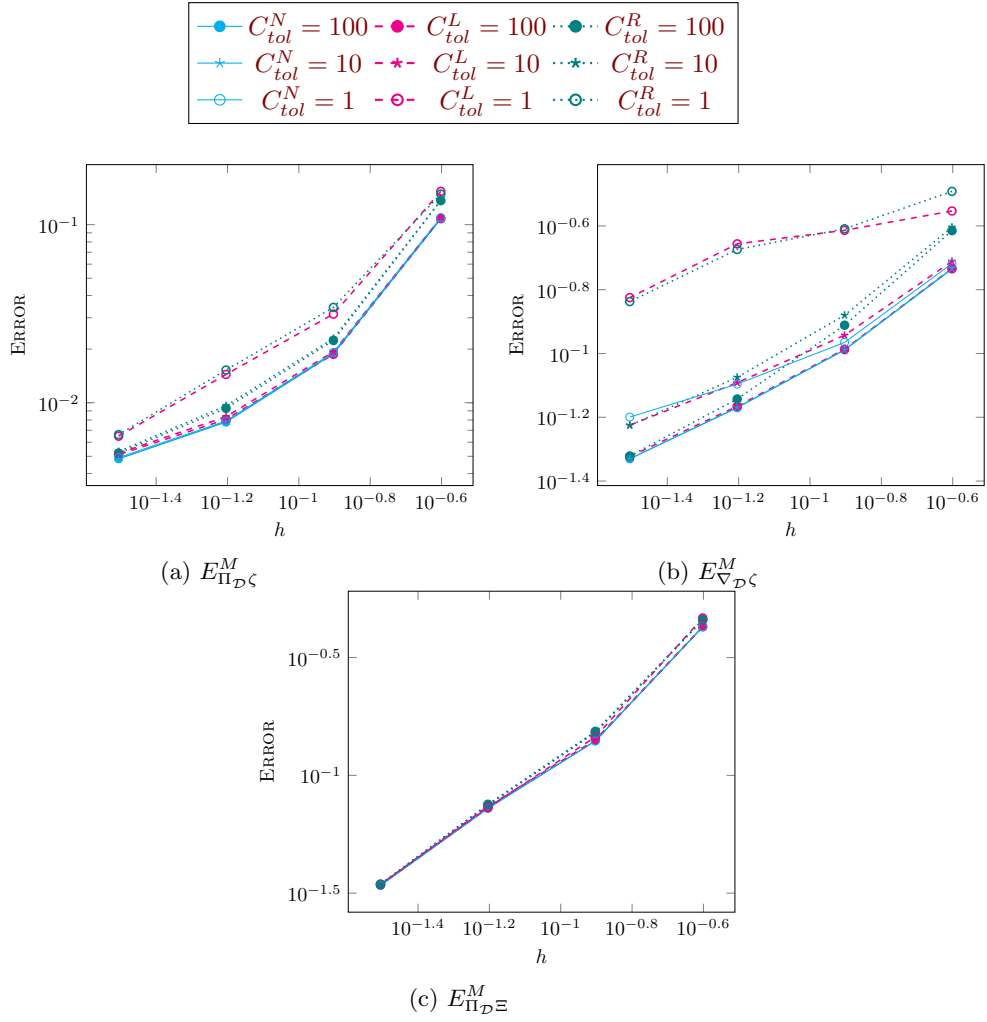
\begin{figure}\centering
  \centerline{ \ref{SEZ}}
  \vspace{0.50cm}
\begin{minipage}{0.45\textwidth}
\begin{tikzpicture}[scale=0.75]
\begin{loglogaxis}[name=egraphSEZ4,legend columns=3,legend to name=SEZ, 
        xlabel=$h$,
         ylabel=\textsc{Error},x post scale =1
         ]         
         \addplot[cyan,mark=*, mark options={solid},] 
        table[x=h, y=ELZ] {chapters/dat_N/T100NME5bm25.dat};
    \addlegendentry{$C_{tol}^N=100$}
        \addplot[magenta,thick,mark=*, mark options={solid},dashed] 
        table[x=h, y=ELZ] {chapters/dat_N/T100LME5bm25.dat};
           \addlegendentry{$C_{tol}^L=100$}
             \addplot[teal,thick,mark=*, mark options={solid},dotted] 
         table[x=h, y=ELZ] {chapters/dat_N/T100RME5bm25.dat};
         \addlegendentry{$C_{tol}^R=100$}
         \addplot[cyan,mark=star, mark options={solid}] 
        table[x=h, y=ELZ] {chapters/dat_N/T10NME5bm25.dat};
         \addlegendentry{$C_{tol}^N=10$}
        \addplot[magenta,thick,mark=star, mark options={solid},dashed] 
        table[x=h, y=ELZ] {chapters/dat_N/T10LME5bm25.dat};
         \addlegendentry{$C_{tol}^L=10$}
             \addplot[teal,thick,mark=star, mark options={solid},dotted] 
         table[x=h, y=ELZ] {chapters/dat_N/T10RME5bm25.dat};
           \addlegendentry{$C_{tol}^R=10$}
          \addplot[cyan,mark=o, mark options={solid},] 
        table[x=h, y=ELZ] {chapters/dat_N/T1NME5bm25.dat};
         \addlegendentry{$C_{tol}^N=1$}
        \addplot[magenta,thick,mark=o, mark options={solid},dashed] 
        table[x=h, y=ELZ] {chapters/dat_N/T1LME5bm25.dat};
         \addlegendentry{$C_{tol}^L=1$}
             \addplot[teal,thick,mark=o, mark options={solid},dotted] 
         table[x=h, y=ELZ] {chapters/dat_N/T1RME5bm25.dat};
    \addlegendentry{$C_{tol}^R=1$}
               \end{loglogaxis}
\end{tikzpicture}
\subcaption{$E^M_{\Pi_\mathcal{D}\zeta}$}
\end{minipage}
 \hskip 30pt
\begin{minipage}{0.45\textwidth}
\begin{tikzpicture}[scale=0.75]
\begin{loglogaxis}[name=egraphSEZ5,
        xlabel=$h$,
         ylabel=\textsc{Error},x post scale =1
         ]         
         \addplot[cyan,mark=*, mark options={solid},] 
        table[x=h, y=EHZ] {chapters/dat_N/T100NME5bm25.dat};
        \addplot[magenta,thick,mark=*, mark options={solid},dashed] 
        table[x=h, y=EHZ] {chapters/dat_N/T100LME5bm25.dat};
             \addplot[teal,thick,mark=*, mark options={solid},dotted] 
         table[x=h, y=EHZ] {chapters/dat_N/T100RME5bm25.dat};
         \addplot[cyan,mark=star, mark options={solid}] 
        table[x=h, y=EHZ] {chapters/dat_N/T10NME5bm25.dat};
        \addplot[magenta,thick,mark=star, mark options={solid},dashed] 
        table[x=h, y=EHZ] {chapters/dat_N/T10LME5bm25.dat};
             \addplot[teal,thick,mark=star, mark options={solid},dotted] 
         table[x=h, y=EHZ] {chapters/dat_N/T10RME5bm25.dat};
          \addplot[cyan,mark=o, mark options={solid},] 
        table[x=h, y=EHZ] {chapters/dat_N/T1NME5bm25.dat};
        \addplot[magenta,thick,mark=o, mark options={solid},dashed] 
        table[x=h, y=EHZ] {chapters/dat_N/T1LME5bm25.dat};
             \addplot[teal,thick,mark=o, mark options={solid},dotted] 
         table[x=h, y=EHZ] {chapters/dat_N/T1RME5bm25.dat};
               \end{loglogaxis}
\end{tikzpicture}
\subcaption{$E^M_{\nabla_\mathcal{D}\zeta}$}
\end{minipage}
 \hskip 30pt
\begin{minipage}{0.45\textwidth}
\begin{tikzpicture}[scale=0.75]
\begin{loglogaxis}[name=egraphSEZ6,
        xlabel=$h$,
         ylabel= \textsc{Error},x post scale =1
         ]         
         \addplot[cyan,mark=*, mark options={solid},] 
        table[x=h, y=EXI] {chapters/dat_N/T100NME5bm25.dat};
        \addplot[magenta,thick,mark=*, mark options={solid},dashed] 
        table[x=h, y=EXI] {chapters/dat_N/T100LME5bm25.dat};
             \addplot[teal,thick,mark=*, mark options={solid},dotted] 
         table[x=h, y=EXI] {chapters/dat_N/T100RME5bm25.dat};
         \addplot[cyan,mark=star, mark options={solid}] 
        table[x=h, y=EXI] {chapters/dat_N/T10NME5bm25.dat};
        \addplot[magenta,thick,mark=star, mark options={solid},dashed] 
        table[x=h, y=EXI] {chapters/dat_N/T10LME5bm25.dat};
             \addplot[teal,thick,mark=star, mark options={solid},dotted] 
         table[x=h, y=EXI] {chapters/dat_N/T10RME5bm25.dat};
          \addplot[cyan,mark=o, mark options={solid},] 
        table[x=h, y=EXI] {chapters/dat_N/T1NME5bm25.dat};
        \addplot[magenta,thick,mark=o, mark options={solid},dashed] 
        table[x=h, y=EXI] {chapters/dat_N/T1LME5bm25.dat};
             \addplot[teal,thick,mark=o, mark options={solid},dotted] 
         table[x=h, y=EXI] {chapters/dat_N/T1RME5bm25.dat};
               \end{loglogaxis}
\end{tikzpicture}
\subcaption{$E^M_{\Pi_\mathcal{D}\Xi}$}
\end{minipage}
\caption{Stochastic case: comparison of various errors for different choices of $C_{tol}$ and $C_{\varepsilon}=1$}
\label{fig:SEgraph_Ezeta}
 \end{figure}%
based on the relative error norms $E^M_{\Pi_\mathcal{D}\zeta}$, $E^M_{\nabla_\mathcal{D}\zeta}$ and $E^M_{\Pi_\mathcal{D}\Xi}$ (for detail methodology, refer to \cite{DRONIOU2024114}, against a reference solution computed on the finest mesh (M5). The errors and norm values are calculated using 25 random Brownian motions. Across all cases, the quantities exhibit convergence, with particular improvement seen in L and R for $\nabla\zeta$, which had previously struggled in the deterministic setting. The N method remains consistent across different values of $C_{\text{tol}}$, while L and R display minor variations. In Figure \ref{fig:Egraph_Ezeta}--(a), although the errors using L and R method are slightly higher for $C_{\text{tol}}=1$, they eventually achieve a similar level of saturation. Figure \ref{fig:Egraph_Ezeta}--(b) shows higher rates of convergence for L and R  with $C_{\text{tol}}=1$ on larger mesh sizes,  with errors significantly decreasing as the mesh is refined. This initial higher error is expected due to the large tolerance chosen. However, there is no significant difference in errors for $E^M_{\Pi_\mathcal{D}\Xi}$ in Figure  \ref{fig:Egraph_Ezeta}--(c). 

The figures in \ref{fig:SVgraph} illustrates the convergence behaviour of various norm values, including the norms of \(\Pi_{\mathcal{D}} \zeta\), \(\nabla_{\mathcal{D}} \zeta\), and \(\Pi_{\mathcal{D}} \Xi\). In all cases of $C_{\text{tol}}$, we observe that the norm values converge towards a constant as the mesh size decreases, demonstrating the reliability of using various tolerance levels and their potential effects on the solution. Overall, there is not a significant difference in the errors and norms values for $C_{\text{tol}}\ge 10$. Additionally, the plots for $\nabla_{\mathcal{D}} \zeta$ shows convergence even for $C_{\text{tol}}=1$, although with a comparatively larger error. Notably, a significant reduction in computational time is observed for the L and R methods, as shown in Figure \ref{fig:scpu} for $C_{\text{tol}}\ge 10$.

These experiments provide a valuable insights for those seeking a reliable alternative to the Newton method. These linearised methods (L and R) do an excellent job, even in the deterministic case, when the tolerance is selected appropriately, offering the facility to tune accuracy and efficiency to specific requirements. In the Stefan problem, the Newton method neither struggles significantly nor diverges, which suggest that under more challenging conditions, the performance of the L and R methods could be more advantageous.%
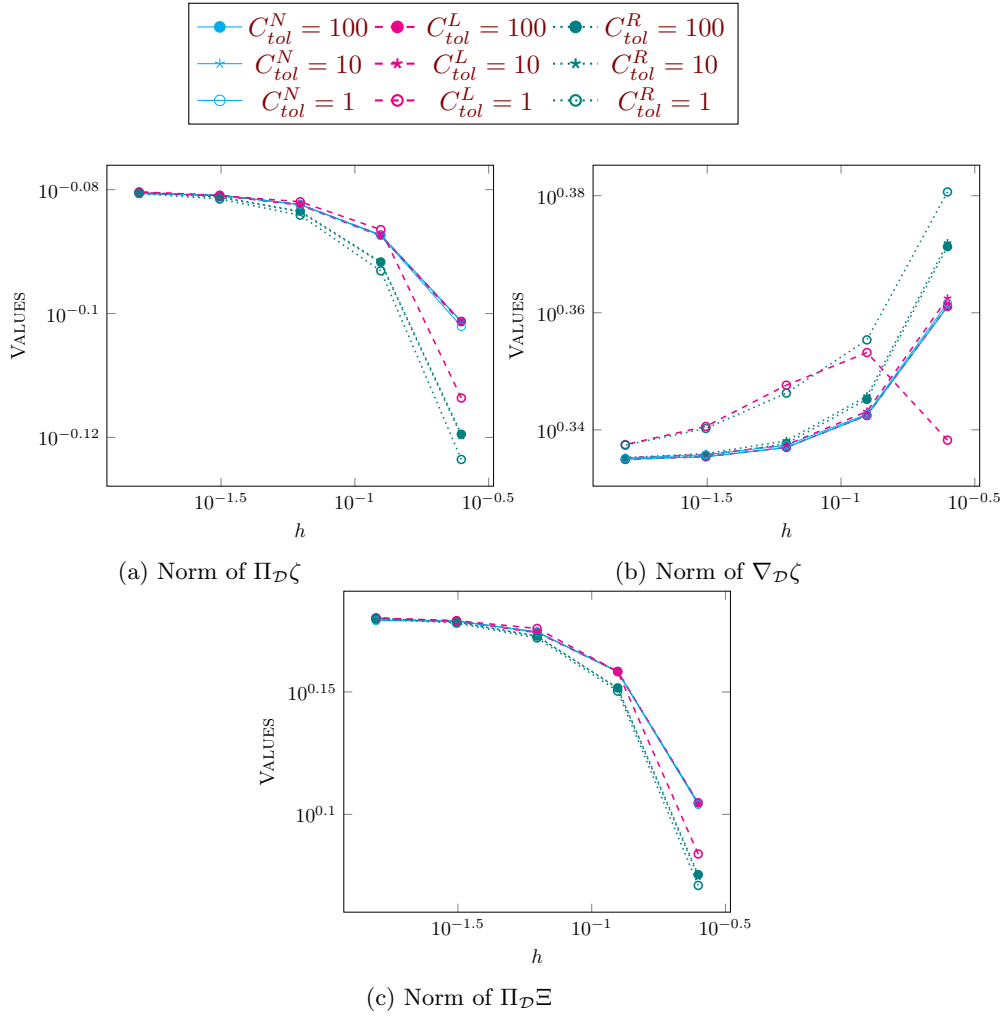
\begin{figure}\centering
\centerline{ \ref{SVZ}}
\vspace{0.50cm}
\begin{minipage}{0.45\textwidth}
\begin{tikzpicture}[scale=0.75]
\begin{loglogaxis}[name=egraphSVZ4,legend columns=3,legend to name=SVZ, 
        xlabel=$h$,
         ylabel=\textsc{Values},x post scale =1
         ]         
         \addplot[cyan,mark=*, mark options={solid},] 
        table[x=h, y=VLZ] {chapters/dat_N/T100NMV5bm25.dat};
    \addlegendentry{$C_{tol}^N=100$}
        \addplot[magenta,thick,mark=*, mark options={solid},dashed] 
        table[x=h, y=VLZ] {chapters/dat_N/T100LMV5bm25.dat};
           \addlegendentry{$C_{tol}^L=100$}
             \addplot[teal,thick,mark=*, mark options={solid},dotted] 
         table[x=h, y=VLZ] {chapters/dat_N/T100RMV5bm25.dat};
         \addlegendentry{$C_{tol}^R=100$}
         \addplot[cyan,mark=star, mark options={solid}] 
        table[x=h, y=VLZ] {chapters/dat_N/T10NMV5bm25.dat};
         \addlegendentry{$C_{tol}^N=10$}
        \addplot[magenta,thick,mark=star, mark options={solid},dashed] 
        table[x=h, y=VLZ] {chapters/dat_N/T10LMV5bm25.dat};
         \addlegendentry{$C_{tol}^L=10$}
             \addplot[teal,thick,mark=star, mark options={solid},dotted] 
         table[x=h, y=VLZ] {chapters/dat_N/T10RMV5bm25.dat};
           \addlegendentry{$C_{tol}^R=10$}
          \addplot[cyan,mark=o, mark options={solid},] 
        table[x=h, y=VLZ] {chapters/dat_N/T1NMV5bm25.dat};
         \addlegendentry{$C_{tol}^N=1$}
        \addplot[magenta,thick,mark=o, mark options={solid},dashed] 
        table[x=h, y=VLZ] {chapters/dat_N/T1LMV5bm25.dat};
         \addlegendentry{$C_{tol}^L=1$}
             \addplot[teal,thick,mark=o, mark options={solid},dotted] 
         table[x=h, y=VLZ] {chapters/dat_N/T1RMV5bm25.dat};
    \addlegendentry{$C_{tol}^R=1$}
               \end{loglogaxis}
\end{tikzpicture}
\subcaption{Norm of $\Pi_{\mathcal{D}}\zeta$}
\end{minipage}
 \hskip 30pt
\begin{minipage}{0.45\textwidth}
\begin{tikzpicture}[scale=0.75]
\begin{loglogaxis}[name=egraphSEZ5,
        xlabel=$h$,
         ylabel=\textsc{Values},x post scale =1
         ]         
         \addplot[cyan,mark=*, mark options={solid},] 
        table[x=h, y=VHZ] {chapters/dat_N/T100NMV5bm25.dat};
        \addplot[magenta,thick,mark=*, mark options={solid},dashed] 
        table[x=h, y=VHZ] {chapters/dat_N/T100LMV5bm25.dat};
             \addplot[teal,thick,mark=*, mark options={solid},dotted] 
         table[x=h, y=VHZ] {chapters/dat_N/T100RMV5bm25.dat};
         \addplot[cyan,mark=star, mark options={solid}] 
        table[x=h, y=VHZ] {chapters/dat_N/T10NMV5bm25.dat};
        \addplot[magenta,thick,mark=star, mark options={solid},dashed] 
        table[x=h, y=VHZ] {chapters/dat_N/T10LMV5bm25.dat};
             \addplot[teal,thick,mark=star, mark options={solid},dotted] 
         table[x=h, y=VHZ] {chapters/dat_N/T10RMV5bm25.dat};
          \addplot[cyan,mark=o, mark options={solid},] 
        table[x=h, y=VHZ] {chapters/dat_N/T1NMV5bm25.dat};
        \addplot[magenta,thick,mark=o, mark options={solid},dashed] 
        table[x=h, y=VHZ] {chapters/dat_N/T1LMV5bm25.dat};
             \addplot[teal,thick,mark=o, mark options={solid},dotted] 
         table[x=h, y=VHZ] {chapters/dat_N/T1RMV5bm25.dat};
               \end{loglogaxis}
\end{tikzpicture}
\subcaption{Norm of $\nabla_{\mathcal{D}}\zeta$}
\end{minipage}
\\
\begin{minipage}{0.45\textwidth}
\begin{tikzpicture}[scale=0.75]
\begin{loglogaxis}[name=egraphSEZ6,
        xlabel=$h$,
         ylabel=\textsc{Values},x post scale =1
         ]         
         \addplot[cyan,mark=*, mark options={solid},] 
        table[x=h, y=VXI] {chapters/dat_N/T100NMV5bm25.dat};
        \addplot[magenta,thick,mark=*, mark options={solid},dashed] 
        table[x=h, y=VXI] {chapters/dat_N/T100LMV5bm25.dat};
             \addplot[teal,thick,mark=*, mark options={solid},dotted] 
         table[x=h, y=VXI] {chapters/dat_N/T100RMV5bm25.dat};
         \addplot[cyan,mark=star, mark options={solid}] 
        table[x=h, y=VXI] {chapters/dat_N/T10NMV5bm25.dat};
        \addplot[magenta,thick,mark=star, mark options={solid},dashed] 
        table[x=h, y=VXI] {chapters/dat_N/T10LMV5bm25.dat};
             \addplot[teal,thick,mark=star, mark options={solid},dotted] 
         table[x=h, y=VXI] {chapters/dat_N/T10RMV5bm25.dat};
          \addplot[cyan,mark=o, mark options={solid},] 
        table[x=h, y=VXI] {chapters/dat_N/T1NMV5bm25.dat};
        \addplot[magenta,thick,mark=o, mark options={solid},dashed] 
        table[x=h, y=VXI] {chapters/dat_N/T1LMV5bm25.dat};
             \addplot[teal,thick,mark=o, mark options={solid},dotted] 
         table[x=h, y=VXI] {chapters/dat_N/T1RMV5bm25.dat};
               \end{loglogaxis}
\end{tikzpicture}
\subcaption{Norm of $\Pi_{\mathcal{D}}\Xi$}
\end{minipage}
\caption{Stochastic case: comparison of various norm values for different choices of $C_{tol}$ and $C_{\varepsilon}=1$}
\label{fig:SVgraph}
 \end{figure}%
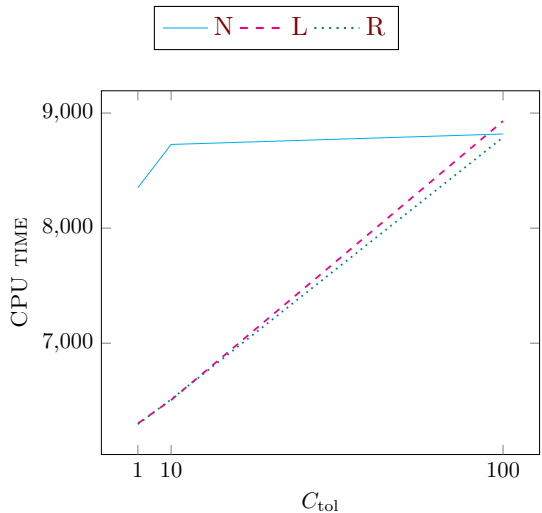
\begin{figure}\centering
%
\centerline{ \ref{scpu}}
\vspace{0.50cm}
\begin{tikzpicture}[scale=0.85]
        \begin{axis}[name=SCPU,legend columns=3,legend to name=scpu, 
        xtick=\empty,
         extra x ticks={1,10,100},
         extra x tick labels={1,10,100},
        xlabel=$C_{\text{tol}}$,
         ylabel=$\textsc{CPU time}$ 
         ]         
        \addplot[cyan] 
        table[x=ctol, y=N] {chapters/dat_N/scpu.dat};
        \addlegendentry{N}
        \addplot[magenta,thick,dashed] 
        table[x=ctol, y=L] {chapters/dat_N/scpu.dat};
         \addlegendentry{L}
         \addplot[teal,thick,dotted] 
         table[x=ctol, y=R] {chapters/dat_N/scpu.dat};
        \addlegendentry{R}
          \end{axis}
    \end{tikzpicture}
    \caption{Stochastic case: CPU times vs.\ $C_{\text{tol}}=\{1,10,100\}$ with $C_{\varepsilon}=1$}
    \label{fig:scpu}
 \end{figure}%

Moreover, the L and R methods excel in the stochastic case, particularly when simulating a problem for a large number of Brownian motions. These methods significantly reduce CPU time as the number of Brownian motions increase, while still maintaining the desired level of accuracy. These tests highlight the need to decrease the tolerance while estimating the gradient of a quantity. 

\section{Conclusion}
We presented a numerical 
{scheme} 
for 
the 
Stefan problem
{, in both deterministic and stochastic cases. The scheme combines the Euler implicit discretisation in time with a gradient method in space. For solving the resulting nonlinear, fully discrete systems, several strategies are studied, including a detailed comparison of the efficiency concerning the error reduction and the computational time.} 
These strategies include the standard Newton method, as well as 
{fixed-point iterative schemes, including or not a regularisation step}. Through numerical sensitivity analysis, we demonstrated how to optimise the choice of the parameters to balance the errors due to the discretisation, the linear iterations and, if applicable, the regularisation. This allows to increase the efficiency of the computations while maintaining the accuracy. The numerical tests provide a guidance for selecting the most appropriate method based on specific circumstances, with a particular emphasis on the advantages of the linearised methods in stochastic cases. Specifically, the L and R- methods proved to be highly effective, especially when simulating problems involving a large number of Brownian motions, significantly reducing the CPU time while preserving the desired accuracy. The findings suggest that decreasing the tolerance in the gradient estimation can further enhance the performance of these methods.

\section*{Acknowledgements}
M.A. Khan acknowledges the support of Ghazi university and higher education commission (HEC), Pakistan. The work of I.S. Pop was supported by the Research Foundation - Flanders (FWO), project G0A9A25N and the German Research Foundation (DFG) through the SFB 1313, project number 327154368.

\printbibliography
\section{Appendix}
We used the following inequalities.
\begin{lemma}\label{lem:zetaLab}
    Let $\zeta$ be a non-decreasing Lipschitz continuous function with Lipschitz constant $L_\zeta\ge0$, such that $0\le\zeta'(s)\le L_\zeta$ for all $s\in\RR$, and let $L\ge L_\zeta/2$. Then for any $a,b\in \RR$, the following inequality holds:
    \bea[eq:zetaLab]
    \left| L(a-b)-\left(\zeta(a)-\zeta(b)\right)\right|\le L \left|a-b\right|.
    \eea
\end{lemma}
\begin{proof}
    By the mean value theorem, there exists some $c\in (a,b)$ such that
    $$
    \zeta(a)-\zeta(b)=\zeta'(c)(a-b).
    $$
    Substituting this into the left-hand side of \eqref{eq:zetaLab}, we have
    $$
    \left| L(a-b)-\left(\zeta(a)-\zeta(b)\right)\right|=\left|L-\zeta'(c)\right|\left|a-b\right|.
    $$
    Since $0\le\zeta'(c)\le L_\zeta$, it follows that $L-L_\zeta\le L-\zeta'(c)\le L$. Taking absolute values, we have
    $$
    \left|L-\zeta'(c)\right|\le \operatorname{max}\left(L, |L-L_\zeta|\right).
    $$
    Now, given $L\ge L_\zeta/2$, it follows that $2L-L_\zeta\ge0$ this ensures
    $L-L_\zeta\ge -L$. Moreover, since $L_\zeta\ge 0$, we also have $L-L_\zeta\le L$. Combining last two inequalities, we conclude that
    $$
    \left|L-\zeta'(c)\right|\le L.
    $$
    Thus the required inequality \eqref{eq:zetaLab} follows. 
\end{proof}
\begin{lemma}\label{lem:ZLab}
    Let $Z$ be non-decreasing Lipschitz continuous function with $0\le L_Z\le Z' (u)$. If $L\ge Z'(u)$ then, for $a,b\in\RR$, the following inequality holds:
    \bea[eq:ZLab]
    |L(a-b)-(Z(a)-Z(b))|\le (L-L_Z)|a-b|.
    \eea
\end{lemma}
\begin{proof}
By the mean value theorem, there exists some $c\in (a,b)$ such that
    $$
    Z(a)-Z(b)=Z'(c)(a-b).
    $$
    Substituting this into the left-hand side of \eqref{eq:ZLab}, we have
    $$
    \left| L(a-b)-\left(Z(a)-Z(b)\right)\right|=\left(L-Z'(c)\right)\left|a-b\right|.
    $$
    Since $L_Z\le Z'(u)$, The conclusion follows from $L-Z'(u)\le L-L_Z.$
\end{proof}
\end{document}